\documentclass{amsart}

\usepackage{amsthm,amsmath}
\usepackage{amssymb, amsfonts}
\usepackage{hyperref}
\usepackage{dsfont}
\usepackage{mathtools}
\usepackage{upgreek}
\usepackage{fancyhdr}
\usepackage{tikz}
\usepackage{chngcntr}
\usepackage{wasysym}
\usepackage{tikz-cd}
\usepackage{color}
\usepackage{ulem}
\usetikzlibrary{trees}

\newtheorem{thm}{Theorem}[section]
\newtheorem{prop}[thm]{Proposition}
\newtheorem{cor}[thm]{Corollary}
\newtheorem{lemma}[thm]{Lemma}
\newtheorem{fact}[thm]{Fact}
\newtheorem{question}[thm]{Question}
\theoremstyle{definition}
\newtheorem{rmk}[thm]{Remark}

\newtheorem{dfn}[thm]{Definition}
\newtheorem{exam}[thm]{Example}
\newtheorem{nota}[thm]{Notation}

\newtheorem{claim}[thm]{Claim}

\newtheorem{theorem}[thm]{Theorem}

\newtheorem{definition}[thm]{Definition}

\newtheorem{remark}[thm]{Remark}

\newtheorem{example}[thm]{Example}
\newtheorem{rem/def}[thm]{Remark/Definition}
\newtheorem{notation}[thm]{Notation}

\counterwithin{figure}{thm}

\def \tree {2{^{<\omega }}}

\def \trek {2{{^{<\kappa}}}}
\def \treo {\omega{^{<\omega}}}
\def \troo {\omega{^{\omega}}}
\def \trec {\kappa{^{<\lambda }}}

\def \tri {\trianglelefteq}
\def \trr {\trianglerighteq}
\def \lex {<_{lex}}

\newcommand{\trn}{%
  \mathrel{\ooalign{$\lneq$\cr\raise.21ex\hbox{$\lhd$}\cr}}}
\newcommand{\trrn}{%
  \mathrel{\ooalign{$\gneq$\cr\raise.21ex\hbox{$\rhd$}\cr}}}
\def \coc {{^{\frown}}}
\def \lr {{\langle \rangle}}
\def \lor {{\langle 0 \rangle}}
\def \llr {{\langle 1 \rangle}}

\def \res {\lceil}
\def \cl {{\rm cl}}

\def \la {\langle}
\def \ra {\rangle}

\def \tp {\textrm{tp}}

\def \age {\textrm{age}}

\def \L {\mathcal{L}}
\def \K {\mathcal{K}}
\def \M {\mathbb{M}}
\def \ac {\text{ac}}

\providecommand{\Th}{\operatorname{Th}}

\def\CL{{\mathcal L}}

\def\BN{{\mathbb N}}
\def\BZ{{\mathbb Z}}

\def\tp{\operatorname{tp}}

\def\x{\bar{x}}
\def\y{\bar{y}}

\def\a{\bar{a}}

\def\Th{\operatorname{Th}}

\def\acl{\operatorname{acl}}
\def\dcl{\operatorname{dcl}}

\def\qftp{\operatorname{qftp}}

\title{On the Antichain Tree Property}
\author[J. Ahn]{JinHoo Ahn$^{\dagger}$}
\address{$\dagger$KIAS, School of Computational Sciences\\ 85 Hoegiro, Dongdaemun-gu\\ 02455, Seoul, South Korea}
\email{jinhooahn@kias.re.kr}

\author[J. Kim]{Joonhee Kim$^{\ddagger}$}
\address{$\ddagger$Yonsei University, Department of Mathematics\\ 50 Yonsei-Ro, Seodaemun-Gu\\ 03722, Seoul, South Korea}
\email{kimjoonhee@yonsei.ac.kr}

\author[J. LEE]{Junguk Lee$^{\ast}$}
\address{$^{\ast}$KAIST, Department of Mathematical Sciences\\ 291 Daehak-ro Yuseong-gu\\        34141, Daejeon, South Korea}
\email{ljwhayo@kaist.ac.kr}

\thanks{$\dagger$The first author is supported by a KIAS Individual Grant (CG079801) at Korea Institute for Advanced Study. $\ddagger$The second author is supported by NRF of Korea grants (2018R1D1A1A02085584, 2021R1A2C1009639), and the Yonsei University Research Fund (Post Doc. Researcher Supporting Program) of 2021 (2021-12-0015). $^{\ast}$The third author is supported by KAIST Advanced Institute for Science-X fellowship.\\ \indent The authors thank the anonymous referee for careful reading and all the comments and suggestions.}

\date {\today}

\begin{document}

\maketitle

\begin{abstract}
In this note, we investigate a new model theoretical tree property, called the antichain tree property (ATP). We develop combinatorial techniques for ATP. First, we show that ATP is always witnessed by a formula in a single free variable, and for formulas, not having ATP is closed under disjunction. Second, we show the equivalence of ATP and $k$-ATP, and provide a criterion for theories to have not ATP (being NATP).

 Using these combinatorial observations, we find algebraic examples of ATP and NATP, including pure groups, pure fields, and valued fields. More precisely, we prove Mekler's construction for groups, Chatzidakis' style criterion for PAC fields, and the AKE-style principle for valued fields preserving NATP. And we give a construction of an antichain tree in the Skolem arithmetic and atomless Boolean algebras.
\end{abstract}

\setcounter{tocdepth}{10}

\section{Introduction}

After Shelah's tree property (TP) was introduced, trees became a common configuration when studying first order theories by focusing on consistency and inconsistency of sets defined by instances of a single formula.
In \cite{DS}, Džamonja and Shelah defined a weakened version of TP called NSOP$_1$ and NSOP$_2$. It is obvious that NSOP$_1$ implies NSOP$_2$ by definitions but the other direction is still unknown at the present. 
In \cite{AK}, the first and second author studied the differences between SOP$_1$ and SOP$_2$ and then introduced a new notion around the tree property called the antichain tree property (ATP).


\begin{dfn}\cite[Definition 4.1]{AK}\label{dfn:antichain_tree_property}
We say a formula $\varphi(x,y)$ has the {\it antichain tree property} (ATP) if there exists a tree indexed set of paremeters $(a_\eta)_{\eta\in 2^{<\omega}}$ such that
\begin{enumerate}
\item[(i)] for any antichain $X$ in $2^{<\omega}$, the set $\{\varphi(x,a_\eta):\eta\in X\}$ is consistent,
\item[(ii)] for any $\eta,\nu\in 2^{<\omega}$, if $\eta\trn\nu$, then $\{\varphi(x,a_\eta),\varphi(x,a_\nu)\}$ is inconsistent.
\end{enumerate}
We say a theory has ATP if there exists a formula having ATP. If a theory does not have ATP, then we say the theory has NATP.
\end{dfn}

 If a theory has ATP, then the theory always has TP$_2$ and SOP$_1$ (\cite[Proposition 4.4, 4.6]{AK}). Thus the class of NTP$_2$ theories and the class of NSOP$_1$ theories are subclasses of the class of NATP theories. Therefore we have the following diagram:

\begin{center}
\begin{tikzcd}
 {\rm DLO}  \arrow{r} & {\rm NIP}  \arrow{r} & {\rm NTP}_2\arrow{r} & {\rm NATP} \\ 
& {\rm stable} \arrow{r} \arrow{u} & {\rm simple}\arrow{r} \arrow{u} & {\rm NSOP}_1 \arrow{u} \arrow{r}  & {\rm NSOP}_2 \arrow{r} & \cdots \\
\end{tikzcd}
\end{center}

We will try to convince the reader that ATP can be a meaningful dividing line, and observe several interesting phenomena around ATP. 

This paper consists of two parts. The first part is finding combinatorial tools that deal with ATP, and the second part is finding algebraic examples of ATP and NATP. As the begining of studying ATP, in Section \ref{sec:Combinatorial properties of ATP}, we develop some combinatorial techniques related to ATP. First, we define a suitable indiscernibility for antichains, and find some modification of the path-collapse lemma which will be called the antichain-collapse lemma. By using it, we prove that ATP is always witnessed by a formula in a single free variable and witnessing NATP is closed under disjunction. Second, we introduce the notion of $k$-ATP for $k\ge 2$ and show the equivalence of ATP and $k$-ATP. Also, we give a combinatorial criterion for theories to be NATP using strong indiscernibility.

 In Section \ref{sec:Algebraic examples}, we give algebraic examples of ATP and NATP. We prove that Mekler's construction for groups, the Chatzidakis criterion\footnote{We at first called `the Chatzidakis-Ramsey criterion', and Ramsey suggested to us `the Chatzidakis criterion' because its originality essentially comes from her work in \cite{C2}.} for PAC fields, and the AKE-style principle for valued fields preserving NATP. Specially, we have an NATP pure group and an NATP valued field which are SOP$_1$ and TP$_2$.

 At the end of the paper, we give a construction of an antichain tree in the theories of Skolem arithmetic and atomless Boolean algebras. They will be good illustrations when we want to imagine what an antichain tree looks like.

 The equation ``NATP = NSOP$_1$ + NTP$_2$" plays an important role in the study of NATP, analogous to the role played by Shelah's equation ``NIP = stability + dense linear order" in the study of NIP and of Chernikov's equation ``NTP$_2$ = simplicity + NIP" ({\it cf.}\;\cite{Che}). From the results covered by this paper, especially algebraic examples of ATP and NATP, we can observe that many of the arguments studying NSOP$_1$ and NTP$_2$ are naturally extended to NATP. 
 
 Interestingly, the principle ``ATP $\perp$ SOP$_2$'' also plays an important role, where the expression means ATP has the opposite nature to SOP$_2$ as a tree property (see Remark \ref{rem:ATP_vs_SOP2} below). For example, it can be seen that some arguments of studying NSOP$_2$ are extended to NATP by exchanging the roles of paths and antichains.

\begin{rmk}\label{rem:ATP_vs_SOP2}\cite[Definition 4.1, 4.2]{AK}
Let $\varphi(x,y)$ be a formula and $(a_\eta)_{\eta\in 2^{<\omega}}$ be a tree indexed set of parameters. 
\begin{enumerate}
\item[(i)] We say $(\varphi,(a_\eta))$ witnesses ATP if for any $X\subseteq\tree$, the set of formulas $\{\varphi(x,a_\eta)\}_{\eta\in X}$ is consistent if and only if $X$ is pairwisely {\it incomparable}.
\item[(ii)] We say $(\varphi,(a_\eta))$ witnesses SOP$_2$ if for any $X\subseteq\tree$, the set of formulas $\{\varphi(x,a_\eta)\}_{\eta\in X}$ is consistent if and only if $X$ is pairwisely {\it comparable}.
\end{enumerate}
\end{rmk}

\section{Preliminaries and notations}\label{sec:Preliminaries and notations}

In this section we summarize some notations and basic facts of the subjects to be covered in our paper. First we recall notations and definitions of tree properties.

\begin{notation}\label{language of trees}
Let $\kappa$ and $\lambda$ be cardinals.
\begin{enumerate}
\item[(i)] By $\kappa^\lambda$, we mean the set of all functions from $\lambda$ to $\kappa$. 
\item[(ii)] By $\trec$, we mean $\bigcup_{\alpha<\lambda}{\kappa^\alpha}$ and call it a tree. If $\kappa=2$, we call it a binary tree. If $\kappa\geq\omega$, then we call it an infinitary tree.
\item[(iii)] By $\emptyset$ or $\lr$, we mean the empty string in $\trec$, which means the empty set (recall that every function can be regarded as a set of ordered pairs).
\end{enumerate}
Let $\eta,\nu\in {\trec}$. 
\begin{enumerate}
\item[(iv)] By $\eta\tri\nu$, we mean $\eta \subseteq \nu$.  If $\eta\tri\nu$ or $\nu\tri\eta$, then we say $\eta$ and $\nu$ are comparable.
\item[(v)] By $\eta\perp\nu$, we mean that $\eta\not\tri\nu$ and $\nu\not\tri\eta$. We say $\eta$ and $\nu$ are incomparable if $\eta\perp\nu$.
\item[(vi)] By $\eta\wedge\nu$, we mean the maximal $\xi\in{\trec}$ such that $\xi\tri\eta$ and $\xi\tri\nu$.
\item[(vii)] By $l(\eta)$, we mean the domain of $\eta$.
\item[(viii)] By $\eta\lex\nu$, we mean that either $\eta\tri\nu$, or $\eta\perp\nu$ and $\eta(l(\eta\wedge\nu))<\nu(l(\eta\wedge\nu))$. 
\item[(ix)] By $\eta\coc\nu$, we mean $\eta\cup\{(i+l(\eta),\nu(i)):i< l(\nu)\}$.
\end{enumerate}
Let $X\subseteq \trec$.
\begin{enumerate}
\item[(x)]
By $\eta\coc X$ and $X\coc\eta$, we mean $\{\eta\coc x:x\in X\}$ and $\{x\coc\eta:x\in X\}$ respectively.
\end{enumerate}
Let $\eta_0,...,\eta_n\in\trec$.
\begin{enumerate}
\item[(xi)] By $\cl(\eta_0,...,\eta_n)$, we mean a tuple $(\eta_0\wedge\eta_0,...,\eta_0\wedge\eta_n,...,\eta_n\wedge\eta_0,...,\eta_n\wedge\eta_n)$.
\item[(xii)] We say a subset $X$ of $\trec$ is an {\it antichain} if the elements of $X$ are pairwise incomparable, {\it i.e.}, $\eta\perp\nu$ for all $\eta,\nu\in X$).  
\end{enumerate}
\end{notation}

\begin{dfn}
Let $\varphi(x,y)$ be an $\mathcal{L}$-formula. 
\begin{enumerate}
\item[(i)] We say $\varphi(x,y)$ has the {\it tree property} (TP) if there exists a tree-indexed set $(a_\eta)_{\eta\in\treo}$ of parameters and $k\in\omega$ such that
\begin{enumerate}
\item[] $\{\varphi(x,a_{\eta\res n})\}_{n\in\omega}$ is consistent for all $\eta\in\troo$ (path consistency), and
\item[] $\{\varphi(x,a_{\eta\coc i})\}_{i\in\omega}$ is $k$-inconsistent for all $\eta\in\treo$, {\it i.e.}, any subset of $\{\varphi(x,a_{\eta\coc i})\}_{i\in\omega}$ of size $k$ is inconsistent.
\end{enumerate}
\item[(ii)] We say $\varphi(x,y)$ has the {\it tree property of the first kind} (TP$_1$) if there is a tree-indexed set $(a_\eta)_{\eta\in\treo}$ of parameters such that
\begin{enumerate}
\item[] $\{\varphi(x,a_{\eta\res n})\}_{n\in\omega}$ is consistent for all $\eta\in\troo$, and
\item[] $\{\varphi(x,a_\eta),\varphi(x,a_\nu)\}$ is inconsistent for all $\eta\perp\nu\in\treo$.
\end{enumerate}
\item[(iii)] We say $\varphi(x,y)$ has the {\it tree property of the second kind} (TP$_2$) if there is an array-indexed set $(a_{i,j})_{i,j\in\omega}$ of parameters such that
\begin{enumerate}
\item[] $\{\varphi(x,a_{n,\eta(n)})\}_{n\in\omega}$ is consistent for all $\eta\in\troo$, and
\item[] $\{\varphi(x,a_{i,j}),\varphi(x,a_{i,k}))\}$ is inconsistent for all $i,j,k\in\omega$ with $j\neq k$.
\end{enumerate}
\item[(iv)] We say $\varphi(x,y)$ has the {\it 1-strong order property} (SOP$_1$) if there is a binary-tree-indexed set $(a_\eta)_{\eta\in\tree}$ of parameters such that
\begin{enumerate}
\item[] $\{\varphi(x,a_{\eta\res n})\}_{n\in\omega}$ is consistent for all $\eta\in2^{\omega}$,
\item[] $\{\varphi(x,a_{\eta\coc 1}),\varphi(x,a_{\eta\coc 0\coc\nu})\}$ is inconsistent for all $\eta,\nu\in\tree$. 
\end{enumerate}
\item[(v)] We say $\varphi(x,y)$ has the {\it 2-strong order property} (SOP$_2$) if there is a binary-tree-indexed set $(a_\eta)_{\eta\in\tree}$ of parameters such that
\begin{enumerate}
\item[] $\{\varphi(x,a_{\eta\res n})\}_{n\in\omega}$ is consistent for all $\eta\in2^\omega$,
\item[] $\{\varphi(x,a_\eta),\varphi(x,a_\nu)\}$ is inconsistent for all $\eta\perp\nu\in\tree$.
\end{enumerate}
\item[(vi)] We say a theory has TP if there is a formula having TP with respect to its monster model of the theory. Sometimes we say that the theory is TP, and we call the theory an TP theory. We define TP$_1$ theory, TP$_2$ theory, SOP$_1$ theory, and SOP$_2$ theory in the same manner.
\item[(vii)] We say a theory is NTP if the theory is not TP, and we call the theory NTP theory. We define NTP$_1$ theory, NTP$_2$ theory, NSOP$_1$ theory, and NSOP$_2$ theory in the same manner.
\end{enumerate}
\end{dfn}

The following facts are well known ({\it cf}. \cite{Con}, \cite{DS}, \cite{KK}, and \cite{She}).

\begin{fact}
\begin{enumerate}
\item[(i)] A theory has TP$_1$ if and only if it has SOP$_2$. 
\item[(ii)] A theory has TP if and only if it has TP$_1$ or TP$_2$.
\item[(iii)] A theory has TP if and only if it has SOP$_1$ or TP$_2$.
\item[(iv)] If a theory has SOP$_2$, then it has SOP$_1$.
\end{enumerate}
\end{fact}

Now we recall the notion of tree indiscernibility and modeling property. We follow the notations in \cite{TT}. Let $\mathcal{L}_0=\{\tri,\lex,\wedge\}$ be a language where $\tri$, $\lex$ are binary relation symbols, and $\wedge$ is a binary function symbol. Then for cardinals $\kappa>1$ and $\lambda$, a tree $\trec$ can be regarded as an $\mathcal{L}_0$-structure whose interpretations of $\tri,\lex,\wedge$ follow Notation \ref{language of trees}.

\begin{dfn}
Let $\overline{\eta}=( \eta_0,...,\eta_n)$ and $\overline{\nu}=( \nu_0,...,\nu_n)$ be finite tuples of $\trec$. 
\begin{enumerate}
\item[(i)] By $\textrm{qftp}_0(\overline{\eta})$, we mean the set of quantifier free $\mathcal{L}_0$-formulas $\varphi(\overline{x})$ such that $\trec\models \varphi(\overline{\eta})$. 
\item[(ii)] By $\overline{\eta}\sim_0\overline{\nu}$, we mean $\textrm{qftp}_0(\overline{\eta})=\textrm{qftp}_0(\overline{\nu})$. We say $\bar{\eta}$ and $\bar{\nu}$ are {\it strongly isomorphic} if $\overline{\eta}\sim_0\overline{\nu}$.
\end{enumerate}
Let $\mathcal{L}$ be a language, $T$ be a complete $\mathcal{L}$-theory, $\mathbb{M}$ be a monster model of $T$, and $( a_\eta )_{\eta\in\trec}$, $( b_\eta )_{\eta\in\trec}$ be tree-indexed sets of parameters from $\mathbb{M}$. For $\overline{\eta}=( \eta_0,...,\eta_n)$, we denote $( a_{\eta_0},...,a_{\eta_n})$ by $\overline{a}_{\overline{\eta}}$. For any finite set of $\mathcal{L}$-formulas $\Delta$ and a type $\Gamma$, we denote $\{\varphi\in\Gamma : \varphi \in \Delta\}$ by $\Gamma_{\Delta}$. By $\overline{a}_{\overline{\eta}}\equiv_{\Delta,A}\overline{b}_{\overline{\nu}}$ we mean $\textrm{tp}_\Delta(\overline{a}_{\overline{\eta}}/A)=\textrm{tp}_\Delta(\overline{b}_{\overline{\nu}}/A)$.
\begin{enumerate}
\item[(iii)] We say $( a_\eta )_{\eta\in\trec}$ is {\it strongly indiscernible} if $\textrm{tp}(\overline{a}_{\overline{\eta}})=\textrm{tp}(\overline{a}_{\overline{\nu}})$ for all $\textrm{qftp}_0(\overline{\eta})=\textrm{qftp}_0(\overline{\nu})$.
\item[(iv)] We say $( b_\eta)_{\eta\in\trec}$ is {\it strongly locally based} on $(a_\eta)_{\eta\in\trec}$ if for all $\overline{\eta}$ and a finite set of $\mathcal{L}$-formulas $\Delta$, there is $\overline{\nu}$ such that $\overline{\eta}\sim_0\overline{\nu}$ and $\overline{b}_{\overline{\eta}}\equiv_\Delta\overline{a}_{\overline{\nu}}$. 
\end{enumerate}
\end{dfn}

\begin{fact}\label{modeling property}
Let a tree-indexed set $( a_\eta )_{\eta\in\treo}$ be given. Then there is a strongly indiscernible $( b_\eta )_{\eta\in\treo}$ which is strongly locally based on $( a_\eta )_{\eta\in\treo}$. 
\end{fact}

\noindent The proof can be found in \cite{KK} and \cite{TT}. The above statement is called {\it the modeling property of strong indiscernibility} (in short, we write it {\it the strong modeling property}).

We also recall some knowledge of Ramsey classes. The following definitions and notations are from \cite{Hod}, \cite{KL}, and \cite{Sco}.

\begin{dfn}
Let $\L$ be a language.
\begin{enumerate}
\item[(i)] Let $M$ be an $\L$-structure. By $age(M)$, we mean the class of all finitely generated $\L$-structures which are isomorphic to a substructure of $M$.  
\item[(ii)] Let $A$ and $C$ be $\L$-structures. By $A$-{\it substruture of }$C$, we mean the class of $\L$-structures $A'\subseteq C$ isomorphic to $A$ and denote it by $\binom{C}{A}$.
\item[(iii)] For each $k\in\omega\backslash\{0\}$, we call a map from $\binom{C}{A}$ to $k$ a {\it k-coloring of }$\binom{C}{A}$.
\item[(iv)] Let $\K$ be a class of $\L$-structures, $A, B, C\in\K$, and $k\in\omega\backslash\{0\}$. By $C\to(B)^A_k$, we mean that for all $k$-colorings $f$ of $\binom{C}{A}$, there exists $B'\subseteq C$ such that $B'$ is $\L$-isomorphic to $B$ and $f\res\binom{B'}{A}$ is constant.
\item[(v)] A class $\K$ of $\L$-structures is called a {\it Ramsey class} if for any $A,B\in\K$ and $k\in\omega\backslash\{0\}$, there exists $C\in\K$ such that $C\to(B)^A_k$.
\item[(vi)] An $\L$-structure $\mathcal{I}$ is said to be {\it locally finite} if for all finite subsets $X$ of $\mathcal{I}$, there is a finite substructure $I_0$ of $\mathcal{I}$ containing $X$.
\end{enumerate}
Now we introduce a more general notion of modeling property.
\begin{enumerate}
\item[(vii)] Let $\L$ and $\L'$ be languages, $T$ be an $\L$-theory, $\M$ be a monster model of $T$. Let $\mathcal I$ be an $\mathcal L'$-structure. For a tuple $\bar i\subseteq \mathcal I$, we write $\qftp_{\mathcal L'}(\bar i)$ for the set of quantifier free $\mathcal L'$-formulas $\varphi(x)$ such that $\mathcal I\models \varphi(\bar i)$. We say an $\mathcal{I}$-indexed set $(a_i)_{i\in \mathcal{I}}\subseteq\M$ is {\it $\L'$-indiscernible} if $\bar{a}_{\bar{i}}\equiv\bar{a}_{\bar{j}}$ for all $\bar{i},\bar{j}\subseteq\mathcal{I}$ with $\qftp_{\mathcal L'}(\bar i)=\qftp_{\mathcal L'}(\bar j)$. 
We say an $\mathcal{I}$-indexed set $(b_i)_{i\in\mathcal{I}}\subseteq\M$ is {\it locally $\L'$-based} on $(a_i)_{i\in \mathcal{I}}\subseteq\M$ if for any finite set $\Phi$ of $\L$-formulas and $\bar{i}\in\mathcal{I}$, there exists $\bar{j}\in\mathcal{I}$ such that $\bar{b}_{\bar{i}}\equiv_\Phi\bar{a}_{\bar{j}}$ and $\qftp_{\mathcal L'}(\bar i)=\qftp_{\mathcal L'}(\bar j)$.
We say $\mathcal{I}$-indexed indiscernible sets have the modeling property with respect to $\L'$ (sometimes we write it $\L'$-modeling property) if for any $(a_i)_{i\in \mathcal{I}}\subseteq\M$, there is an $\L'$-indiscernible $(b_i)_{i\in\mathcal{I}}\subseteq\M$ which is locally $\L'$-based on $(a_i)_{i\in \mathcal{I}}$.
\end{enumerate}
\end{dfn}

\noindent The following facts, proved by Scow in \cite{Sco}, play an important role in Section \ref{sec:Combinatorial properties of ATP}.

\begin{dfn}{\cite[Definition 2.3 (2)]{Sco}}
Let $\L$ be a language and $\mathcal{I}$ be an $\L$-structure. We say $\mathcal{I}$ is \texttt{qfi} if for any complete quantifier free type $q(\bar{x})$ realized in $\mathcal{I}$, there is a quantifier-free formula $\theta_q(\bar{x})$ such that $\Th(\mathcal{I})_\forall\cup\theta_q(\bar{x})\vdash q(\bar{x})$.
\end{dfn}

\begin{fact}\label{scow's proposition}{\rm \cite[Proposition 1 (2)]{Sco}}
In the case that $\mathcal{I}$ is a structure in a finite language and is locally finite, then $\mathcal{I}$ is \texttt{qfi}.
\end{fact}

\begin{fact}\label{scow's theorem}{\rm \cite[Theorem 3.12]{Sco}} Let $\mathcal{I}$ be a \texttt{qfi}, locally finite structure in a language $\L$ with a relation $<$ linearly ordering $I$. Then $\mathcal{I}$-indexed indiscernible sets have the modeling property if and only if $\age(\mathcal{I})$ is a Ramsey class.
\end{fact}

\section{Combinatorial properties of ATP}\label{sec:Combinatorial properties of ATP}
In this section, we develop several combinatorial tools to study ATP. Unlike the case of dealing with paths in trees, it is not so easy to obtain homogeneity when dealing with antichains. Thus we introduce a suitable notion of indiscernibility for antichains and prove a kind of lifting lemma which is related to this indiscernibility (Lemma \ref{lem:lifting}). By using this, we prove that if a theory has ATP, then there is a witness of ATP in a single free variable (Theorem \ref{one-varable}), and that NATP is preserved under taking disjunction of formulas (Lemma \ref{lem:disjunction_NATP}). We also prove the equivalence of $k$-ATP and ATP at the level of theories (Lemma \ref{lem:k-ATP=ATP}), and find a criterion for a theory to have ATP (Theorem \ref{thm:indisc_ext_NATP}).

\subsection{Obtaining a witness of ATP in a single free variable}

In this subsection, we show that if a theory has ATP, then there exists a formula in a single free variable which witnesses ATP. Following the notation of \cite{Ram}, let us call this one-variable theorem for ATP. The one-variable theorems for SOP$_1$ and TP$_2$ are proved by Ramsey in \cite{Ram}, and by Chernikov in \cite{Che} respectively. Chernikov and Ramsey also proved the one-variable theorem for SOP$_2$ in \cite[Corollary 4.11]{CR} by using ``Path-Collapse Lemma'' which says:

\begin{fact}{\rm \cite[Lemma 4.6]{CR} (Path Collapse)} Suppose $\kappa$ is an infinite cardinal, $(a_\eta)_{\eta\in\trek}$ is a strongly indiscernible tree over a set of parameters $C$ and, moreover, $(a_{0^\alpha})_{0<\alpha<\omega}$ is order indiscernible over $cC$. Let 
\[
p(y;\bar{z})=\tp(c;(a_{0\coc0^\gamma})_{\gamma<\kappa}/C).
\]
Then if 
\[
p(y:(a_{0\coc0^\gamma})_{\gamma<\kappa})\cup p(y:(a_{1\coc0^\gamma})_{\gamma<\kappa})
\]
is not consistent, then $T$ has SOP$_2$, witnessed by a formula with free variables $y$.
\end{fact}

In order to obtain a witness of ATP in a single free variable, we find an appropriate statement which is similar to the path-collapse lemma. In short, we find a condition for the consistency of a partial type $p(y;\bar{a})\cup p(y;\bar{b})$, where $\bar{a}$, $\bar{b}$ are strongly isomorphic antichains in the same tree which is strongly indiscernible over $c$, and $p(y;\bar{z})=\tp(c;\bar{a})$.

To do this, we first develop some techniques to control the homogeneity of antichains in tree-indexed sets. We introduce an appropriate language which describes the relation between elements of an antichain and show the modeling property with respect to the language. And by using this we show that there is a realization $c$ of the set of formula $\{\varphi(x,a_\eta)\}_{\eta\in X}$, where $X$ is an antichain such that $(a_\eta)_{\eta\in X}$ is strongly indiscernible over $c$. 

\begin{dfn}
Let $I$ be a linearly ordered set whose cardinality is larger than $1$, $\alpha$ be an ordinal, and $X\subseteq I^{<\alpha}$ be an antichain. 
We say $X$ is {\it universal} if for each finite antichain $Y$ in $I^{<\alpha}$, there exists $X_0\subseteq X$ such that $Y\sim_0 X_0$.
\end{dfn}

\begin{exam}
If $\lambda$ is a limit ordinal, then $\omega^\lambda$ and $2^\lambda$ are universal antichains. We can construct universal antichains without using limit ordinals. In $2^{<\omega}$, $\bar{X}=\bigcup_{i\in\omega}(1^i\coc0\coc X_i)$ is also an universal antichain, where $\{X_i\}_{i\in\omega}$ is an enumeration of all finite antichains in $2^{<\omega}$. 
\end{exam}

\begin{nota}
Let $\mathcal{L}_\delta=\{\Delta, \lex\}$ be a language where $\Delta$ is a quarternary relation symbol.
\end{nota}

For a limit ordinal $\kappa$, $\omega^{\leq\kappa}$ can be regarded as an $\mathcal{L}_0$-structure as usual. For each $X\subseteq\omega^{\leq\kappa}$, if $X$ is an antichain, then we can regard $X$ as an $\mathcal{L}_\delta$-structure by interpreting $\Delta$ by
\begin{center}
$\Delta(\eta_0,\eta_1,\eta_2,\eta_3)$ if and only if $\eta_0\wedge\eta_1\tri\eta_2\wedge\eta_3$,
\end{center}
where $\tri$ and $\wedge$ are as usual in $\omega^{\leq\kappa}$.
From now, by $\delta$-indiscernibility, $\delta$-basedness, and $\delta$-modeling property, we mean the $\L_{\delta}$-indiscernibility, the $\L_\delta$-basedness, and the $\L_\delta$-modeling property respectively. For $L_\delta$-structure $X$ and $\bar{\eta},\bar{\nu}\subseteq X$, by $\bar{\eta}\sim_\delta\bar{\nu}$ we mean that the quantifier free $\L_\delta$-types of $\bar{\eta}$ and $\bar{\nu}$ are the same.

\begin{lemma}\label{ss->ll}
Let $\kappa$ be a limit ordinal, and $\bar{\eta}=(\eta_0,...,\eta_n)$ and $\bar{\nu}=(\nu_0,...,\nu_n)$ be tuples in $\omega^\kappa$.
If $\bar{\eta}\sim_{\delta}\bar{\nu}$, then $\cl(\bar{\eta})\sim_0\cl(\bar{\nu})$ in $\omega^{\le\kappa}$. Conversely, if $\cl(\bar{\eta})\sim_0\cl(\bar{\nu})$ in $\omega^{\leq\kappa}$, then $\bar{\eta}\sim_{\delta}\bar{\nu}$. Therefore, if $(a_\eta)_{\eta\in \omega^{\leq\kappa}}$ is strongly indiscernible, then $(a_\eta)_{\eta\in \omega^\kappa}$ is $\delta$-indiscernible.
\begin{proof}
Suppose $\bar{\eta}\sim_\delta\bar{\nu}$. To show $\cl(\bar{\eta})\sim_0\cl(\bar{\nu})$, it is enough to show that
\begin{center}
$\eta_i\wedge\eta_j\tri\eta_k\wedge\eta_l$ if and only if $\nu_i\wedge\nu_j\tri\nu_k\wedge\nu_l$,\\
$\eta_i\wedge\eta_j\lex\eta_k\wedge\eta_l$ if and only if $\nu_i\wedge\nu_j\lex\nu_k\wedge\nu_l$,
\end{center}
for all $i,j,k,l\leq n$. The first condition is given by $\Delta$. Thus it is enough to show that the second condition holds. Suppose $\eta_i\wedge\eta_j\lex\eta_k\wedge\eta_l$. If $\eta_i\wedge\eta_j\tri\eta_k\wedge\eta_l$, then by $\Delta$, we have $\nu_i\wedge\nu_j\tri\nu_k\wedge\nu_l$ and hence $\nu_i\wedge\nu_j\lex\nu_k\wedge\nu_l$. If $\eta_i\wedge\eta_j\not\tri\eta_k\wedge\eta_l$, then $\eta_i\wedge\eta_j\not\trr\eta_k\wedge\eta_l$ and $\eta_i,\eta_j\lex\eta_k,\eta_l$ thus we have $\nu_i\wedge\nu_j\not\tri\nu_k\wedge\nu_l$, $\nu_i\wedge\nu_j\not\trr\nu_k\wedge\nu_l$, and $\nu_i,\nu_j\lex\nu_k,\nu_l$. Therefore $\nu_i\wedge\nu_j\lex\nu_k\wedge\nu_l$. 

Conversely, we need to show that $\cl(\bar{\eta})\sim_0\cl(\bar{\nu})$ implies $\bar{\eta}\sim_\delta\bar{\nu}$. But this implication is trivial because $\cl(\bar{\eta})$, $\cl(\bar{\nu})$ are meet-closures of $\bar{\eta}$, $\bar{\nu}$, and $\mathcal{L}_0$ contains $\tri$ which defines $\Delta$.
\end{proof}
\end{lemma}

\begin{rmk}\label{modeling properties of Q^omega}
Let $\kappa$ be a limit ordinal. Then the strong modeling property for $\omega^{\leq\kappa}$ holds by Fact \ref{scow's proposition} and Fact \ref{scow's theorem} since $\age(\omega^{\leq\kappa})=\age(\omega^{<\kappa})$. 
\end{rmk}

\begin{lemma}
$\omega^\kappa$ has the $\delta$-modeling property for each limit ordinal $\kappa$.
\begin{proof}
Let $\kappa$ be a limit ordinal and assume that $(a_\eta)_{\eta\in\omega^\kappa}$ is given. We want to find $\delta$-indiscernible $(b_\eta)_{\eta\in\omega^\kappa}$ which is $\delta$-based on $(a_\eta)_{\eta\in\omega^\kappa}$. 

First we complete $(a_\eta)_{\eta\in\omega^{\leq\kappa}}$ from $(a_\eta)_{\eta\in\omega^\kappa}$ by defining $a_\eta=a_{\eta\coc 0^{\kappa'}}$ for each $\eta\in\omega^{<\kappa}$, where $|\eta|+\kappa'=\kappa$. By applying the strong modeling property on $(a_\eta)_{\eta\in\omega^{\leq\kappa}}$, we obtain a strongly indiscernible tree $(b_\eta)_{\eta\in\omega^{\leq\kappa}}$ which is strongly locally based on $(a_\eta)_{\eta\in\omega^{\leq\omega}}$.

Now we show that $(b_\eta)_{\eta\in\omega^\kappa}$ is $\delta$-indiscernible and $\delta$-based on $(a_\eta)_{\eta\in\omega^\kappa}$. By Lemma \ref{ss->ll}, $(b_\eta)_{\eta\in\omega^\kappa}$ is $\delta$-indiscernible. Choose any $\eta_0,...,\eta_n\in\omega^\kappa$ and a finite set of formulas, say $\Phi$ . Then there exist $\nu_0,...,\nu_n\in\omega^{\leq\kappa}$ such that $\cl(\eta_0,...,\eta_n)\sim_0\cl(\nu_0,...,\nu_n)$ and $\bar{b}_{\cl(\bar{\eta})}\equiv_\Phi\bar{a}_{\cl(\bar{\nu})}$. In particular, we have
\[
b_{\eta_0}...b_{\eta_n}\equiv_\Phi a_{\nu_0}...a_{\nu_n}=a_{\nu'_0}...a_{\nu'_n},
\] 
where $\nu'_i=\nu_i$ if $\nu_i\in\omega^\kappa$, and $\nu'_i=\nu_i\coc0^{\kappa'}$ if $\nu_i\in\omega^{<\kappa}$, where $|\nu_i|+\kappa'=\kappa$. Since $\cl(\bar{\eta})\sim_0\cl(\bar{\nu})\sim_0\cl(\bar{\nu'})$, we have $\bar{\eta}\sim_\delta\bar{\nu'}$. Thus $(b_\eta)_{\eta\in\omega^\kappa}$ is $\delta$-based on $(a_\eta)_{\eta\in\omega^\kappa}$.
\end{proof}
\end{lemma}

Now we show the lifting property of $\delta$-indiscernibility.

\begin{dfn}
Suppose $(a_\eta)_{\eta\in\omega^{\leq\kappa}}$ is strongly indiscernible over a set of parameters $D$ and $X$ is an antichain in $\omega^{<\kappa}$. Let $\Sigma(z)$ be a partial type over $D(a_\eta)_{\eta\in X}$. We say $\Sigma(z)$ is {\it closed under $\sim_\delta$ ($\delta$-isomorphism) in $(a_\eta)_{\eta\in X}$ over $D$} if $\varphi(z,\bar{a}_{\bar{\eta}},\bar{d})\in\Sigma(z)$ and $\bar{\eta}\sim_\delta\bar{\nu}$ always imply $\varphi(z,\bar{a}_{\bar{\nu}},\bar{d})\in\Sigma(z)$, for each $\varphi(z,\bar{y},\bar{d})\in\L(D)$ and $\bar{\eta},\bar{\nu}\in X$.
\end{dfn}

\begin{lemma}\label{lem:lifting}
Assume that $\kappa$ is a limit ordinal and $(a_\eta)_{\eta\in\omega^{\leq\kappa}}$ is strongly indiscernible over a set of parameters $D$. Let $X$ be a universal anichain in $\omega^{\leq\kappa}$. Let $\Sigma(z)$ be a consistent partial type over $D(a_\eta)_{\eta\in X}$ which is closed under $\sim_\delta$ in $(a_\eta)_{\eta\in X}$ over $D$. Then there exists $b$ such that:
\begin{enumerate}
\item[(i)] $b\models\Sigma(z)$,
\item[(ii)] $(a_\eta)_{\eta\in X}$ is $\delta$-indiscernible over $bD$.
\end{enumerate}
\end{lemma}
\begin{proof}
Let $\Gamma(z)$ be
\[
\{\psi(a_{\eta_0},...,a_{\eta_n},z)\leftrightarrow\psi(a_{\nu_0},...,a_{\nu_n},z):\psi\in\mathcal{L}(D), (\eta_0,...,\eta_n)\sim_\delta(\nu_0,...,\nu_n)\subseteq X\}.
\]
By adapting the proof scheme of \cite[Claim 4.12]{Scow}, we show that $\Sigma(z)\cup\Gamma(z)$ is consistent. By compactness, it is enough to show that $\Sigma_0\cup\Gamma_0$ is consistent for finite subsets $\Sigma_0$ of $\Sigma$ and $\Gamma_0$ of $\Gamma$.

Choose any $c\models\Sigma(z)$. Since $\Gamma_0$ is finite, there is a finite set $$\Lambda=\{(\varphi_0(\bar x_{\bar \mu^0};z),p_0),\ldots,(\varphi_m(\x_{\bar \mu^m};z),p_m)\},$$ where $\x_{\mu^i}:=(x_{\mu_0^i},\ldots,x_{\mu_{n_i}^i})$ and $p_i=\tp^{qf}_{\CL_{\delta}}(\mu_0^i,\ldots,\mu_{n_i}^i)$ for each $i\le m$, satisfying that there are finite tuples $\bar \eta^i$ and $\bar \nu^i$ of elements of $X$ for each $i\le m$ such that

\begin{align*}
\Gamma_0&= \{\varphi_i(\a_{\bar\eta^i};z)\leftrightarrow \varphi_i(\a_{\bar\nu^i};z):\bar\eta^i,\bar\nu^i\models p_i,i\le m\}\\
&\subset \{\varphi_i(\a_{\bar\eta};z)\leftrightarrow\varphi_i(\a_{\bar\nu};z):(\varphi_i,p_i)\in \Lambda,\bar\eta,\bar\nu\models p_i\}.
\end{align*}
\noindent Let $Z$ be the substructure of $X$ generated by $\{\bar \eta^0,...,\bar{\eta}^m\}$.  Let $Y_{0}$ be the finite substructure of $X$ generated by 
\[
\{\bar\eta^i,\bar\nu^i:i\le m\}\;\cup\;\text{dom}(\Sigma_0).
\]
Note that $Z=\{\bar \eta^0,...,\bar{\eta}^m\}$ and $Y_{0}=\{\bar\eta^i,\bar\nu^i:i\le m\}\;\cup\;\text{dom}(\Sigma_0)$ since $\L_\delta$ has no function symbol.
Since $\age(X)=\age(\omega^{\omega})$ has the Ramsey property, there is $Y_1\in\age(X)$ such that $Y_1\rightarrow (Y_{0})^Z_{2^{m+1}}$. We may assume $Y_1\subseteq X$.
Consider a coloring $c:\binom{Y_1}{Z}\rightarrow 2^{m+1}$ given as follows: For $\la\bar{\xi}^0,...,\bar{\xi}^m\ra\in \binom{Y_1}{Z}$, $c(\la\bar{\xi}^0,...,\bar{\xi}^m\ra)=\la\varepsilon^0,...,\varepsilon^m\ra$ where
$$
\varepsilon^i=\begin{cases}
0&\mbox{if }\models \varphi_i(\a_{\bar{\xi}^i};c)\\
1&\mbox{if }\models \neg\varphi_i(\a_{\bar{\xi}^i};c)
\end{cases}
$$ for each $i\leq m$.
Since $Y_1\rightarrow(Y_{0})^Z_{2^{m+1}}$, there is a substructure $Y_2$ of $Y_1$ isomorphic to $Y_{0}$ such that $c$ is monochromatic on $\binom{Y_2}{Z}$. 
Since $(a_{\eta})_{\eta\in X}$ is $\delta$-indiscernible over $D$, there is an $D$-automorphism $\sigma$ with $\sigma\left((a_{\eta})_{\eta\in Y_2}\right)=(a_{\eta'})_{\eta'\in Y_{0}}$. Since $\Sigma$ is closed under $\sim_\delta$ and $c\models \Sigma$, $\sigma(c)\models \Sigma_0$ and $\sigma(c)\models\Sigma_0\cup \Gamma_0$.
\end{proof}

To show the one-variable theorem for ATP, we need to modify the path-collapse lemma for the purpose of dealing with antichains and ATP. We first see how to construct maximal antichains.

\begin{dfn}
An antichain $X\subseteq 2^{<\kappa}$ is said to be {\it maximal} if there is no antichain $Y\subseteq 2^{<\kappa}$ containing $X$ properly.
\end{dfn}

\begin{rmk}\label{number of maximal antichains}
Let $\alpha_n$ denotes the number of all maximal antichains in $2^{<n}$. Then $\alpha_0=0$ and $\alpha_{n+1}=\alpha_n^2+1$ for each $n\in\omega$. 
\begin{proof}
Let $\{X_i\}_{i\in\alpha_n}$ be the set of all maximal antichains in $2^{<n}$. Then 
\[
\{(\lor\coc X_i)\cup(\llr\coc X_j):i,j<\alpha_n\}\cup\{\emptyset\}
\]
is the set of all maximal antichains in $2^{<n+1}$. Thus $\alpha_{n+1}=\alpha_n^2+1$ for each $n\in\omega$.
\end{proof}
\end{rmk}

\noindent In Remark \ref{number of maximal antichains}, to obtain all maximal antichains in $2^{<n+1}$, there are two main processes. The first process is to take the products of two copies of all maximal antichains in $2^{<n}$. This process is reflected by ``$\{(\lor\coc X_i)\cup(\llr\coc X_j):i,j<\alpha_n\}$'' and ``$\alpha_n^2$'' in Remark \ref{number of maximal antichains}. The second process is to add one more maximal antichain $\{\emptyset\}$ to $\{(\lor\coc X_i)\cup(\llr\coc X_j):i,j<\alpha_n\}$, where $\{\emptyset\}$ is located below all antichains constructed in the first process. This completes the set of all maximal antichains in $2^{<n+1}$. The second process is reflected by ``$\cup\{\emptyset\}$'' and ``$+1$''.

To continue, we introduce the notion of collapsible family of antichains.

\begin{dfn}\label{def:collapsible_family}
Let $\kappa$ be a limit ordinal and $X_0,...,X_n\subseteq \omega^{<\kappa}$ be universal antichains with $|X_0|=...=|X_n|$ and $X_0\sim_0...\sim_0 X_n$. Let us consider the following condition.
\begin{enumerate}
\item[($\ast$)] For any set $C\subseteq\mathbb{M}$, $b\in\mathbb{M}$, tree indexed set $(a_\eta)_{\eta\in \omega^{<\kappa}}$ which is strongly indiscernible over $C$, and $i\leq n$, if $(a_\eta)_{\eta\in X_i}$ is $\delta$-indiscernible over $bC$, then $\bigcup_{j\leq n}p(y,(a_\eta)_{\eta\in X_j})$ is consistent where $p(y,\bar{z})=\tp(b,(a_\eta)_{\eta\in X_i}/C)$ or there is a formula with free variable $y$ which witnesses ATP.
\end{enumerate}
If $X_0,...,X_n$ satisfies ($\ast$), then we say they are {\it collapsible}. By a {\it collapsible family}, we mean a set of universal antichains which are collapsible.
\end{dfn}

\begin{rmk} 
Let $X_0,...,X_n$ be collapsible and $X_0...X_n\sim_0 X'_0...X'_n$. Then $X'_0,...,X'_n$ are also collapsible.
\end{rmk}

\noindent Now we prove two antichain-collapse lemmas (Lemma \ref{collapsing is preserved by product} and Lemma \ref{meet closure antichain}), which reflect the first and second processes in Remark \ref{number of maximal antichains}. The first antichain-collapse lemma claims that the product of two collapsible families is also collapsible. 

\begin{lemma}\label{collapsing is preserved by product}{\rm (The 1st Antichain Collapse)} 
Let $\kappa$ be a sufficiently large cardinal, and $\{X_0,$ $...,$ $X_n\}$ be a collapsible family in $\omega^{<\kappa}$. Then for any $\nu,\xi\in\omega^{<\kappa}$ with $\nu\perp\xi$ and $\nu\lex\xi$, 
\[
\{X^\nu_i\cup X^\xi_j:i,j\leq n \}
\]
is a collapsible family, where $X^\nu_i=\nu\coc X_i$ and $X^\xi_j=\xi\coc X_j$ for each $i,j\leq n$.
\begin{proof}
Clearly $X^\nu_i\cup X^\xi_j$ is an universal antichain for each $i,j\leq n$, and $X^\nu_i\cup X^\xi_j\sim_0 X^\nu_{i'}\cup X^\xi_{j'}$ for all $i,i',j,j'\leq n$. Suppose $C\subseteq\mathbb{M}$, $b\in\mathbb{M}$, and $(a_\eta)_{\eta\in\omega^{<\kappa}}$ is a strongly indiscernible tree over $C$. Assume $(a_\eta)_{\eta\in X^\nu_{i_0}\cup X^\xi_{j_0}}$ is $\delta$-indiscernible over $bC$ for some $i_0,j_0\leq n$. Let
\[
p(y,\bar{z},\bar{u})=\tp(b,(a_\eta)_{\eta\in X^\nu_{i_0}},(a_\eta)_{\eta\in X^\xi_{j_0}}/C).
\]
We first show that $\bigcup_{j\leq n}p(y,(a_\eta)_{\eta\in X^\nu_{i_0}},(a_\eta)_{\eta\in X^\xi_j})$ is consistent or there is a formula with free variable $y$ which witnesses ATP. Since $(a_\eta)_{\eta\in X^\nu_{i_0}\cup X^\xi_{j_0}}$ is $\delta$-indiscernible over $bC$, it is true that $(a_\eta)_{\eta\in X^\xi_{j_0}}$ is $\delta$-indiscernible over $bC(a_\eta)_{\eta\in X^\nu_{i_0}}$. Note that $(a_\eta)_{\eta\in\xi\coc\omega^{<\kappa}}$ is strongly indiscernible over $C(a_\eta)_{\eta\in X^\nu_{i_0}}$. Let 
\[
q(y,\bar{u})=\tp(b,(a_\eta)_{\eta\in X^\xi_{j_0}}/C(a_\eta)_{\eta\in X^\nu_{i_0}}).
\]
Then $\bigcup_{j\leq n}q(y,(a_\eta)_{\eta\in X^\xi_j})$ is consistent or there is a formula with free variable $y$ which witnesses ATP. Note that
\[
p(y, (a_\eta)_{\eta\in X^\nu_{i_0}}, (a_\eta)_{\eta\in X^\xi_j})\subseteq q(y,(a_\eta)_{\eta\in X^\xi_j})
\]
for each $j\leq n$. Thus $\bigcup_{j\leq n}p(y, (a_\eta)_{\eta\in X^\nu_{i_0}}, (a_\eta)_{\eta\in X^\xi_j})$ is consistent or there is a formula with free variable $y$ which witnesses ATP.

Now we show that $\bigcup_{i\leq n}\bigcup_{j\leq n}p(y,(a_\eta)_{\eta\in X^\nu_{i}},(a_\eta)_{\eta\in X^\xi_j})$ is consistent or there is a formula with free variable $y$ which witnesses ATP. Assume that there is no formula with free variable $y$ which witnesses ATP. Then $\bigcup_{j\leq n}q(y,(a_\eta)_{\eta\in X^\xi_j})$ is consistent, as we observed above. Note that $(a_\eta)_{\eta\in X^\nu_{i_0}}$ is $\delta$-indiscernible over $bC(a_\eta)_{\eta\in X^\xi_{j_0}}$.
Thus $\bigcup_{j\leq n}q(y,(a_\eta)_{\eta\in X^\xi_j})$ is closed under $\sim_\delta$ in $(a_\eta)_{\eta\in X^\nu_{i_0}}$ over $C(a_\eta)_{\eta\in X^\xi_{\leq n}}$. Since $(a_\eta)_{\eta\in\nu\coc\omega^{<\kappa}}$ is strongly indiscernible over $C(a_\eta)_{\eta\in X^\xi_{\leq n}}$, we can apply Lemma \ref{lem:lifting} to obtain $b'$ such that
\begin{center}
$\tp(b'/C(a_\eta)_{\eta\in X^\nu_{i_0}\cup X^\xi_{\leq n}})\supseteq\bigcup_{j\leq n}q(y,(a_\eta)_{\eta\in X^\xi_j})$,
\end{center}
and $(a_\eta)_{\eta\in X^\nu_{i_0}}$ is $\delta$-indiscernible over $b'C(a_\eta)_{\eta\in X^\xi_{\leq n}}$. Let 
\[
r(y,\bar{z})=\tp(b',(a_\eta)_{\eta\in X^\nu_{i_0}}/C(a_\eta)_{\eta\in X^\xi_{j\leq n}}).
\]
Then $\bigcup_{i\leq n}r(y,(a_\eta)_{\eta\in X^\nu_i})$ is consistent since we assume there is no ATP witness with free variable $y$. It is clear that 
\[
\bigcup_{j\leq n}p(y,(a_\eta)_{\eta\in X^\nu_i},(a_\eta)_{\eta\in X^\xi_j})\subseteq r(y,(a_\eta)_{\eta\in X^\nu_i})
\]
for each $i\leq n$. Thus $\bigcup_{i\leq n}\bigcup_{j\leq n}p(y,(a_\eta)_{\eta\in X^\nu_{i}},(a_\eta)_{\eta\in X^\xi_j})$ is consistent. This completes the proof.
\end{proof}
\end{lemma}

\noindent Then we get the following proposition by applying the first antichain-collapse lemma repeatedly.
\begin{prop}\label{repeating}
Let $\{X_0,...,X_n\}$ be a collapsible family in $\omega^{<\kappa}$. Let $X=\{\xi_0\lex ...\lex\xi_m\}$ be a finite antichain in $\omega^{<\kappa}$. Then,
\begin{center}
$\{\bigcup_{k\leq m} X^{\xi_k}_i : i\leq n\}$
\end{center}
is a collapsible family, where $X^{\xi_k}_i=\xi_k\coc X_i$ for each $i\leq n$ and $k\leq m$.
\end{prop}

Now we prove the second antichain-collapse lemma which says that every collapsible family $\mathcal{F}'$ can be embedded into a collapsible family $\mathcal{F}\cup\{X\}$ where $\mathcal{F}\sim_0 \mathcal{F}'$ and $X$ is an antichain located below $\mathcal{F}$.

\nopagebreak

\begin{lemma}\label{meet closure antichain}{\rm(The 2nd Antichain Collapse)} 
\nopagebreak Let $\kappa$ be a sufficiently large cardinal, and $\{X'_0,$ $... ,$ $X'_n\}$ be a collapsible family in $\omega^{<\kappa}$. Then there is a collapsible family $\{X_0,$ $... ,$ $X_{n+1}\}$ in $\omega^{<\kappa}$ which satisfies that there are $\chi\in X_{n+1}$ and $\chi'\trr\chi$ such that $X_0=\chi'\coc X'_0, ...,$ $X_n=\chi'\coc X'_n$.
\begin{proof}
Choose any $\nu'\in X'_0$. Then there is $\nu\in\omega^{<\kappa}$ such that $\nu'\tri\nu$ and $\nu\not\tri\eta$ for all $\eta\in X'_i$, for all $i\leq n$. Let $X_i=\nu\coc\lor\coc X'_i$ for each $i\leq n$ and $X_{n+1}=X'_0$. Then $X_0,...,X_{n+1}$ are universal antichains with $|X_0|=...=|X_{n+1}|$ and $X_0\sim_0...\sim_0 X_{n+1}$.

We show $X_0,...,X_{n+1}$ are collapsible. Suppose $C\subseteq\mathbb{M}$, $b\in\mathbb{M}$, and $(a_\eta)_{\eta\in\omega^{<\kappa}}$ is a strongly indiscernible tree over $C$. First we assume $(a_\eta)_{\eta\in X_{n+1}}$ is $\delta$-indiscernible over $bC$ and show $\bigcup_{i\leq n+1}p(y,(a_\eta)_{\eta\in X_i})$ is consistent or there is a formula with free variable $y$ witnessing ATP, where $p(y,\bar{z})=\tp(b,(a_\eta)_{\eta\in X_{n+1}}/C)$. Let us assume $\bigcup_{i\leq n+1}p(y,(a_\eta)_{\eta\in X_i})$ is inconsistent. Since $X_{n+1}\sim_0 X_0$, there exists $b'$ such that $\tp(b',(a_\eta)_{\eta\in X_0}/C)=\tp(b,(a_\eta)_{\eta\in X_{n+1}}/C)$. Hence, if $\bigcup_{i\leq n}p(y,(a_\eta)_{\eta\in X_i})$ is inconsistent, then there is a formula with free variable $y$ which witnesses ATP because $(a_\eta)_{\eta\in X_0}$ is $\delta$-indiscernible over $b'C$, and $X_0,...,X_n$ are collapsible. Therefore we suppose that $\bigcup_{i\leq n}p(y,(a_\eta)_{\eta\in X_i})$ is consistent. Then there exists a formula $\varphi(y,\bar{z})$ in $p(y,\bar{z})$ such that $\bigwedge_{i\leq n+1}\varphi(y,\bar{a}_{\bar{\eta}^i})$ is not consistent and $\bigwedge_{i\leq n}\varphi(y,\bar{a}_{\bar{\eta}^i})$ is consistent, where $\bar{\eta}^i$ $=$ $(\eta^i_0,...,\eta^i_k)$ $\in X_i$ for each $i\leq n+1$ and $\bar{\eta}^i\sim_0 \bar{\eta}^j$ for all $i,j\leq n+1$. For each $i\leq n$ and $l\leq k$, there exists $\xi^i_l\in X'_i$ such that $\nu\coc\lor\coc\xi^i_l=\eta^i_l$. We may assume $\xi^{0}_{l}=\eta^{n+1}_{l}$ for all ${l}\leq k$. 
We construct a binary tree $(\bar{a}^{\varepsilon})_{\varepsilon\in\tree}$ by $\bar{a}^{\varepsilon}=(\bar{a}_{\chi_\varepsilon\coc\bar{\xi}^0},...,\bar{a}_{\chi_\varepsilon\coc\bar{\xi}^n})$ where 
\[
\begin{array}{rl}
\chi_\emptyset & \!\!\!\! =\emptyset,\\
\chi_{\varepsilon\coc\lor} & \!\!\!\! = \chi_\varepsilon\coc\nu\coc\lor,\\
\chi_{\varepsilon\coc\llr} & \!\!\!\! = \chi_\varepsilon\coc\nu\coc\llr.
\end{array}
\]
Let $\psi(y,\bar{z}_0,...,\bar{z}_n)=\bigwedge_{i\leq n}\varphi(y,\bar{z}_i)$. Clearly $\{\psi(y,\bar{a}^\varepsilon),\psi(y,\bar{a}^{\varepsilon'})\}$ is inconsistent for all $\varepsilon\trn\varepsilon'$. To show $\psi$ witnesses ATP with $(\bar{a}^\varepsilon)_{\varepsilon\in\tree}$, it is enough to show that $\{\psi(y,\bar{a}^\varepsilon)\}_{\varepsilon\in 2^m}$ is consistent for each $m\in\omega$. 
By Proposition \ref{repeating}, we have that 
\begin{center}
$\{\bigcup_{\varepsilon\in 2^m}\chi_\varepsilon\coc X'_i:i\leq n\}$
\end{center}
is a collapsible family. Note that $\bigwedge_{\varepsilon\in2^m}\varphi(y,\bar{a}_{\chi_\varepsilon\coc\bar{\xi}^0})$ is consistent since 
\[
\varphi(y,\bar{z})\in\tp(b,(a_\eta)_{\eta\in X_{n+1}}/C),
\]
$(a_\eta)_{\eta\in X_{n+1}}$ is $\delta$-indiscernible over $bC$, $(\chi_\varepsilon\coc\bar{\xi}^0)_{\varepsilon\in2^m}$ is a finite antichain, and $X_{n+1}$ is universal. Let $\Sigma(y)$ be
\begin{center}
$\{\bigwedge_{\varepsilon\in 2^m}\varphi(y,\bar{a}_{\chi'_\varepsilon}):\la\chi_\varepsilon\coc\bar{\xi}^0 )_{\varepsilon\in2^m}\sim_\delta(\chi'_\varepsilon)_{\varepsilon\in2^m}\in \bigcup_{\varepsilon\in2^m}\chi_\varepsilon\coc X'_0\}$.
\end{center}
Then by the universality of $X_{n+1}$ again, $\Sigma(y)$ is consistent. Clearly $\Sigma(y)$ is closed under $\sim_\delta$ in $(a_\eta)_{\eta\in \bigcup_{\varepsilon\in2^m}\chi_\varepsilon\coc X'_0}$ over $C$. Thus we can apply Lemma \ref{lem:lifting} to obtain $b''\models\bigwedge_{\varepsilon\in2^m}\varphi(y,\bar{a}_{\chi_\varepsilon\coc\bar{\xi}^0})$ such that $\bigcup_{\varepsilon\in 2^m}\chi_\varepsilon\coc X'_0$ is $\delta$-indiscernible over $b''C$. Thus $\bigcup_{i\leq n}q(y,(a_\eta)_{\eta\in \bigcup_{\varepsilon\in 2^m}\chi_\varepsilon\coc X'_i})$ is consistent where 
\[
q(y,\bar{u})=\tp(b'',(a_\eta)_{\eta\in \bigcup_{\varepsilon\in 2^m}\chi_\varepsilon\coc X'_0}).
\]
In particular, $\{\psi(y,\bar{a}^\varepsilon)\}_{\varepsilon\in 2^m}$ is consistent.  Thus there is a formula with free variable $y$ witnessing ATP.

Now we assume $(a_\eta)_{\eta\in X_{i_0}}$ is $\delta$-indiscernible over $bC$ for some $i_0\leq n$, and show that $\bigcup_{i\leq n+1}p(y,(a_\eta)_{\eta\in X_i})$ is consistent or there is a formula with free variable $y$ witnessing ATP, where $p(y,\bar{z})=\tp(b,(a_\eta)_{\eta\in X_{i_0}})$. Since $X_{i_0}\sim_0 X_{n+1}$, there exists $b'$ such that $\tp(b,(a_\eta)_{\eta\in X_{i_0}})=\tp(b',(a_\eta)_{\eta\in X_{n+1}})$ and $(a_\eta)_{\eta\in X_{n+1}}$ is $\delta$-indiscernible over $b'C$. As we observed above, $\bigcup_{i\leq n+1}p(y,(a_\eta)_{\eta\in X_i})$ is consistent or there is a formula with free variable $y$ witnessing ATP. Hence $X_0,...,X_{n+1}$ are collapsible, as desired.
\end{proof}
\end{lemma}

Now we are ready to prove the main statement of this subsection. As in Remark \ref{number of maximal antichains}, $\alpha_m$ denotes the number of all maximal antichains in $2^{<m}$.

\begin{thm}\label{one-varable}
If there exists a witness of ATP, then there exists a witness of ATP in a single free variable.
\begin{proof}
Let $\varphi(x,y;z)$ be a witness of ATP with strongly indiscernible $(a_\eta)_{\eta\in\omega}^{<\kappa}$, where $|a_\eta|=|z|$ for all $\eta\in\omega^{<\kappa}$, $|y|=1$, and $\kappa$ is sufficiently large.
First we construct a collapsible family $\{X_{m,i}\}_{i<\alpha_m}$ and find $f_m:2^{<m}\to\omega^{<\kappa}$ for each $m<\omega\backslash\{0\}$ which satisfy 
\begin{enumerate}
\item[(i)] for each maximal antichain $Y\subseteq 2^{<m}$, there exists a unique $i<\alpha_m$ such that $f_m(Y)\subseteq X_{m,i}$,
\item[(ii)] for any distinct maximal antichains $Y,Y'\subseteq 2^{<m}$, if $f_m(Y)\subseteq X_{m,i}$ and $f_m(Y')\subseteq X_{m,j}$, then $i\neq j$,
\item[(iii)] $f_m$ is a strong embedding, which means $\bar{\eta}\sim_0 f_m(\bar{\eta})$ for all $\bar{\eta}\in2^{<m}$.
\end{enumerate}
Let $X_{1,0}=\omega^\omega$ and $f_1(\emptyset)=0^\omega\in\omega^\omega$. Then $\{X_{1,i}\}_{i<\alpha_1}$ and $f_1$ satisfy the conditions. Let us assume we have constructed $\{X_{m,i}\}_{i<\alpha_m}$ and $f_m:2^{<m}\to\omega^{<\kappa}$. By the first antichain-collapse lemma, 
\[
\Gamma=\{X^\lor_{m,i}\cup X^\llr_{m,j}:i,j<\alpha_m\}
\]
is a collapsible family, where $X^\lor_{m,i}=\lor\coc X_{m,i}$, $X^\llr_{m,j}=\llr\coc X_{m,j}$ for each $i,j\leq \alpha_m$.

Note that $|\Gamma|=\alpha_m^2$ and let $\{X'_{m,i}\}_{i<\alpha_m^2}$ be an enumeration of $\Gamma$. By the second antichain-collapse lemma, there is a collapsible family $\{X_{m+1,i}\}_{i<\alpha_{m+1}}$, which satisfies $X_{m+1,i}=\chi'\coc X'_{m,i}$ for all $i<\alpha_m^2$, for some $\chi\in X_{m+1,\alpha_m^2}$, and $\chi'\trr\chi$. Let $f_{m+1}$ be a map from $2^{<m+1}$ to $\omega^{<\kappa}$ which is defined by
\[
f_{m+1}(\eta)=
\left\{
	\begin{array}{ll}
		\chi  & \mbox{if } \eta=\emptyset\\
		\chi'\coc\lor\coc f_m(\eta') & \mbox{if } \eta=\lor\coc\eta'\\
		\chi'\coc\llr\coc f_m(\eta') & \mbox{if } \eta=\llr\coc\eta'.\\
	\end{array}
\right.
\]
Then $\{X_{m+1,i}\}_{i<\alpha_{m+1}}$ and $f_{m+1}$ satisfy (i),(ii), and (iii).

We first show that for each $m<\omega$, there exists $b_m$ with $|b_m|=|y|$ such that $\varphi(x;y,z)$ satisfies the conditions of ATP with $(b_ma_{f_m(\eta)})_{\eta\in2^{<m}}$. Choose any $i_0<\alpha_m$. Then there exists $b'$ such that $\{\varphi(x,b',a_\eta)\}_{\eta\in X_{m,i_0}}$ is consistent and $(a_\eta)_{\eta\in X_{m,i_0}}$ is $\delta$-indiscernible over $b'$, by Lemma \ref{lem:lifting}. Since $\{X_{m,i}\}_{i<\alpha_m}$ is collapsible, $\bigcup_{i<\alpha_m}p(y,(a_\eta)_{\eta\in X_{m,i}})$ is consistent where $p(y,\bar{z})=\tp(b',(a_\eta)_{\eta\in X_{m,i_0}})$. Choose any $b_m\models\bigcup_{i<\alpha_m}p(y,(a_\eta)_{\eta\in X_{m,i}})$. Then $\varphi(x;y,z)$ witnesses ATP with $(b_ma_{f_m(\eta)})_{\eta\in2^{<m}}$.

So, by compactness, there exist $b$ and a strong embedding $f:2^{<\omega}\to\omega^{<\kappa}$ such that $\varphi(x;y,z)$ witnesses ATP with $(ba_{f(\eta)})_{\eta\in2^{<\omega}}$. By repeating this process we can reduce the number of free variables in the witness of ATP, and finally we obtain a witness of ATP in a single free variable.
\end{proof}
\end{thm}

Next, we show that NATP is preserved under taking disjunction of formulas.

\begin{lemma}\label{lem:disjunction_NATP}
For formulas $\varphi(x,y)$ and $\psi(x,z)$, if $\varphi\vee\psi$ witnesses ATP, then $\varphi$ witnesses ATP or $\psi$ witnesses ATP.
\begin{proof}
Suppose $\varphi(x,y)\vee\psi(x,z)$ witnesses ATP with strongly indiscernible $(a_\eta b_\eta)_{\eta\in\omega^{<\kappa}}$, where $\kappa$ is sufficiently large. Let $c\models\{\varphi(x,a_\eta)\vee\psi(x,b_\eta)\}_{\eta\in\omega^\omega}$. For each $n\in\omega$, choose any $B_n\subseteq\omega^\omega$ such that $B_n\sim_0 2^n\subseteq \omega^{<\kappa}$. Since $\omega^\omega$ is a Ramsey class with respect to $\L_\delta$, there exists $C_n\subseteq \omega^\omega$ such that $C_n\to(B_n)^A_2$ for each $n\in\omega$, where $A=\{0^{\omega}\}$. Let $f_n:C_n\to 2$ be a coloring which is defined by
\[
f_n(\eta)=
\left\{
	\begin{array}{ll}
		0 & \mbox{if } \models\varphi(c,a_\eta)\\
		1 & \mbox{otherwise}.\\
	\end{array}
\right.
\]
Then there exist $B'_n\subseteq C_n$ and ${k_n}<2$ such that $B'_n\sim_\delta B_n$ and $f_n(\eta)={k_n}$ for all $\eta\in B'_n$. By the pigeon hole principle, there exist ${k}<2$ and an infinite subset $I$ of $\omega$ such that ${k_i=k}$ for all ${i}\in I$. We may assume ${k=0}$.

Now we show that $\varphi$ witnesses ATP with $(a_\eta)_{\eta\in\omega^{<\kappa}}$. If $\eta\trn\nu$, then clearly $\{\varphi(x,a_\eta),$ $\varphi(x,a_\nu)\}$ is inconsistent since $\{\varphi(x,a_\eta)\vee\psi(x,b_\eta),\varphi(x,a_\nu)\vee\psi(x,b_\eta)\}$ is inconsistent. Let $X\subseteq \omega^{<\kappa}$ be an antichain. By compactness we may assume $X$ is finite. Then there exist $n\in\omega$ and $X'\subseteq 2^n$ such that $X\sim_0 X'$. Since $I$ is infinite, we may assume $n\in I$. Since ${k_n=k=0}$, the set $\{\varphi(x,a_\eta)\}_{\eta\in B'_n}$ is consistent. Since $B'_n\sim_\delta B_n\sim_0 2^n$, the set $\{\varphi(x,a_\eta)\}_{\eta\in 2^n}$ is consistent. In particular $\{\varphi(x,a_\eta)\}_{\eta\in X'}$ is consistent, and hence $\{\varphi(x,a_\eta)\}_{\eta\in X}$ is consistent. This shows that $\varphi$ witnesses ATP and completes the proof.
\end{proof}
\end{lemma}

\noindent Hence, if a theory has quantifier elimination, then to show the theory is NATP we only need to check if there is no witness of ATP in a single free variable, which is the conjunction of atomic formulas and negation of atomic formulas.

\subsection{k-Antichain tree property}

Consider a tree $I^{<\kappa}$ for an ordered set $I$ and a cardinal $\kappa$.

\begin{definition}\label{def:ATP}
Let $T$ be a complete theory. Let $\varphi(x;y)$ be a formula and let $k\ge 2$ be a positive integer. We say that $\varphi(x;y)$ has {the} \textit{$k$-antichain tree property ($k$-ATP)} if there is a tree of parameters $(a_{\eta})_{\eta\in 2^{<\omega}}$ such that 
\begin{itemize}
    \item For any antichain $A\subset 2^{<\omega}$, $\{\varphi(x;a_{\eta}):\eta\in A\}$ is consistent.
    \item For any pairwise comparable distinct elements $\eta_0,\ldots,\eta_{k-1}\in 2^{<\omega}$, $\{\varphi(x;a_{\eta_i}):i<k\}$ is inconsistent.
\end{itemize}
Note that $2$-ATP is ATP. We say that $T$ has $k$-ATP if there is a formula having $k$-ATP. 
\end{definition}

\begin{lemma}\label{lem:k-ATP=ATP}
A complete theory $T$ has $k$-ATP for some $k\ge 2$ if and only if $T$ has ATP
\end{lemma}
\begin{proof}
 Clearly, for any $2\leq k<k'\in\omega$, $k$-ATP implies $k'$-ATP. Thus it is enough to show that for each positive integer $k\geq2$, if $T$ has $(k+1)$-ATP, then it has $k$-ATP.\par 
 Suppose $\varphi(x;y)$ and $(a_\eta)_{\eta\in 2^{<\omega}}$ witness $(k+1)$-ATP for some $2\leq k<\omega$. By the strong modeling property, we may assume $(a_\eta)_{\eta\in 2^{<\omega}}$ is strongly indiscernible.
 
 Let $(a^{(m)}_\eta)_{\eta\in2^{\omega}}$ and $(b_\eta)_{\eta\in 2^{<\omega}}$ be the $m$-fold fattening and $k$-fold elongation of $(a_\eta)_{\eta\in 2^{<\omega}}$ respectively (\cite[Definition 2.6 (3), (5)]{CR}):
 For each $\eta\in2^{<\omega}$ and $n>0$, $a^{(0)}_\eta=a_\eta$ and $a^{(n+1)}_\eta=(a^{(n)}_{0^\frown\eta},a^{(n)}_{1^\frown\eta})$;
 For each $\eta\in2^{\omega}$, $b_\eta$ is a tuple $(a_{\eta'},\ldots,a_{\eta'^\frown0^{k-1}})$ where $\eta'$ is a tuple of length $k(l(\eta)-1)+1$ in $2^{<\omega}$ such that $\eta'(i)=\eta(i/k)$ if $k$ is divisible by $i$ and $\eta'(i)=0$ otherwise;
 
 Define $K_m$ to be the set $\{\nu^\frown0^i:\nu\in2^m,i<k\}$. Either $\bigwedge_{\eta\in K_m}\varphi(x;a_\eta)$ is inconsistent for some $m<\omega$ or consistent for all $m<\omega$.
 If the first case holds, let $m$ be the least one satisfying the condition and $\psi(x;\Bar{y}):=\bigwedge_{i<2^m}\varphi(x,y_i)$. Then $\psi(x,\Bar{y})$ with the $m$-fold fattening $(a^{(m)}_\eta)_{\eta\in2^{\omega}}$ of $(a_\eta)_{\eta\in 2^{<\omega}}$ witnesses $k$-ATP.
 If the second case holds, let $\psi(x;\Bar{y}):=\bigwedge_{i<k}\varphi(x,y_i)$. Then $k$-fold elongation $(b_\eta)_{\eta\in 2^{<\omega}}$ of $(a_\eta)_{\eta\in 2^{<\omega}}$ and $\psi(x;\Bar{y})$ witnesses 2-ATP. Thus in any case, $T$ has $k$-ATP.
\end{proof}

Now, we want to check that the logic in the proof of \cite[Lemma 2.16]{DS} can be applied when we color not only the binary tree $2^{<\kappa}$ but also $I^{<\kappa}$ for some arbitrary countable ordered set $I$. 
We also examine how large the trees or antichain sets can be embedded as a monochromatic subset.

\begin{notation}
Let $I$ be an ordered set. Let $f:I^{<\kappa}\rightarrow X$ be a function. We write $$f(\nu\lhd):=\{x\in X:\exists \mu\in I^{<\kappa}(\nu\lhd\mu\wedge f(\mu)=x)\}.$$
\end{notation}

\begin{remark}\label{rem:tree_monotone_decreasing_coloring}
Let $\kappa$ be a cardinal and $I$ be a countable ordered set. Let $f:I^{<\kappa}\rightarrow X$ be an arbitrary function and let $c:X\rightarrow \lambda$ be a coloring with $\lambda<cf(\kappa)$. There exist $\eta^*\in I^{<\kappa}$ and a color $\theta\in \lambda$ such that for all $\nu \in I^{<\kappa}$ with $\eta^*\trianglelefteq \nu$, there is $\rho \in I^{<\kappa}$ such that
	\begin{itemize}
	\item $\nu\trianglelefteq\rho$,
	\item for all $i\in I$, the set $\{x\in X:x\in f(\rho^{\frown}i\lhd)\wedge c(x)=\theta\}$ is non-empty.
\end{itemize}
\end{remark}
\begin{proof}
Suppose not. 
For each $i<\lambda$, we inductively choose $\nu_i\in I^{<\kappa}$ as follows:

\medskip

\noindent{\it (Initial)} Find $\nu\in I^{<\kappa}$ such that $\{x\in X:x\in f(\nu^{\frown}j\lhd)\wedge c(x)=0\}=\emptyset$ for some $j\in I$. Put $\nu_0:=\nu^\frown j$.

\medskip

\noindent{\it(Successor)} Assume $\nu_i$ is given. Find $\nu\rhd \nu_i$ such that $\{x\in X:x\in f(\nu^{\frown}j\lhd)\wedge c(x)=i+1\}=\emptyset$ for some $j\in I$. Put $\nu_{i+1}:=\nu^\frown j$.

\medskip

\noindent{\it(Limit)} Let $\mu$ be a limit ordinal and assume $\nu_i$ is given for each $i<\mu$. Let $\nu':=\bigcup_{i<\mu}\nu_i$. Find $\nu\rhd \nu'$ such that $\{x\in X:x\in f(\nu^{\frown}j\lhd)\wedge c(x)=\mu\}=\emptyset$ for some $j\in I$. Put $\nu_{\mu}:=\nu^{\frown }j$.

\medskip

Let $\nu^*:=\bigcup_{i<\lambda}\nu_i$. Note that $f(\nu^*)\in \bigcap_{i<\lambda}f(\nu_i\lhd)$ and so $c(f(\nu^*))\neq i$ for all $i<\lambda$, a contradiction.
\end{proof}

\begin{cor}\label{rem:tree_2kappa_coloring}
Let $\kappa$ and $\lambda$ be infinite cardinals and $\lambda<cf(\kappa)$.
\begin{enumerate}
    \item[(a)] Let $f:2^{<\kappa}\rightarrow X$ be an arbitrary function and $c:X\rightarrow \lambda$ be a coloring map. Then there is a monochromatic subset $S\subseteq2^{<\kappa}$ which is strongly isomorphic to $2^{<cf(\kappa)}$.
    \item[(b)] Let $f:2^{\kappa}\rightarrow X$ be an arbitrary function and $c:X\rightarrow \lambda$ be a coloring map. Then there is a monochromatic subset $S\subseteq2^\kappa$ such that for any $k<\omega$, there exists some tuple in $S$ strongly isomorphic to the lexicographic enumeration of $2^k$.
\end{enumerate}

\end{cor}
\begin{proof}
(a) By remark \ref{rem:tree_monotone_decreasing_coloring}, we can obtain $\eta^\ast\in2^{<\kappa}$ and $\theta\in \lambda$ such that for all $\nu \in 2^{<\kappa}$ with $\eta^*\trianglelefteq \nu$, there is $\rho \in 2^{<\kappa}$ such that
	\begin{itemize}
	\item $\nu\trianglelefteq\rho$,
	\item for all $i<2$, the set $\{x\in X:x\in f(\rho^{\frown}i\lhd)\wedge c(x)=\theta\}$ is non-empty.
\end{itemize}
We inductively define $(\xi_\eta)_{\eta\in2^{<\alpha}}$ for each $\alpha<cf(\kappa)$ as follows:

\medskip

\noindent{\it (Initial)} Take $\nu=\eta^\ast$ and let $\rho_{\langle\rangle}$ be a node in $2^{<\kappa}$ satisfying the conditions above. Then there is $\xi_{\langle\rangle}\in2^{<\kappa}$ with $\rho_{\langle\rangle}\unlhd\xi_{\langle\rangle}$ such that $c(f(\xi_{\langle\rangle}))=\theta$. Fix $\xi_{\langle\rangle}$.

\medskip

\noindent{\it(Successor)} Suppose $(\xi_\eta)_{\eta\in2^{\leq\alpha}}$ is already defined. To define $(\xi_\eta)_{\eta\in2^{\leq\alpha +1}}$, it is enough to find
     $\xi_{\eta^\frown i}$ for each $\eta\in 2^\alpha$, $i<2$. Take $\nu=\xi_{\eta}^\frown i$ and find $\rho_{\eta^\frown i}$ same as in the initial part. Then there is $\xi_{\eta^\frown i}\in2^{<\kappa}$ with $\rho_{\eta^\frown i}\unlhd\xi_{\eta^\frown i}$ such that $c(f(\xi_{\eta^\frown i}))=\theta$. Fix $\xi_{\eta^\frown i}$.

\medskip

\noindent{\it(Limit)} Let $\mu$ be a limit ordinal and suppose $(\xi_\eta)_{\eta\in2^{\leq\mu}}$ is already defined. To find $\xi_\eta$ with $\eta\in 2^\mu$, take $\nu=\bigcup_{\tau\unlhd\eta}\xi_\tau$ and follow the same process. Note $\bigcup_{\tau\unlhd\eta}\xi_\tau$ is in $2^{<\kappa}$ as $\lambda<cf(\kappa)$.

\medskip

Then $S=\{\xi_\eta: \eta\in2^{<cf(\kappa)}\}$ is the desired set.

\medskip

(b)
Let $f':2^{<\kappa}\rightarrow 2^\kappa$ be the function such that $f'(\eta)=\eta^\frown 0^\kappa$ and let $c'$ be the composite function of $f$ and $c$. 
Applying (a) on $f'$ and $c'$, we obtain $S'\subseteq2^{<\kappa}$ and $\theta\in \lambda$ such that $S'\sim_0 2^{<cf(\kappa)}$ and for all $\nu \in S'$, $c'(f'(\nu))=\theta$. Then $S:=\{f'(\nu):\nu\in S'\}$ is the desired set since $cf(\kappa)$ is infinite.
\end{proof}

\begin{remark}\label{rem:arbitrary_configuration_embedding_ATP}
Let $\varphi(x;y)$ be a formula having ATP. For any cardinal $\mu$ and any antichain $A\subset 2^{<\mu}$, there is a cardinal $\mu'$ and there is a strongly indiscernible $(a_{\eta})_{\eta\in 2^{\le \mu'}}$ witnessing ATP of $\varphi$, a parameter $b$, and $S\subset 2^{\mu'}$ such that
\begin{itemize}
    \item for all $\eta'\in S$, $\models \varphi(b,a_{\eta'})$,
    \item for all $\eta_1,\eta_2\in S$, $a_{\eta_1}\equiv_b a_{\eta_2}$,
    \item there is a strongly isomorphism $\iota:A\rightarrow S$.
\end{itemize}
\end{remark}
\begin{proof}
Fix a cardinal $\mu$ and an antichain $A\subset 2^{<\mu}$. Let $\mu'$ be a cardinal with $cf(\mu')>\mu+2^{|T|}$. Take a strongly indiscernible $(a_{\eta})_{\eta\in 2^{\le\mu'}}$ witnessing ATP of $\varphi$. 

Let $b\models \{\varphi(x;a_{\eta}):\eta\in 2^{\mu'}\}$ and let $X:=2^{\mu'}$. Consider a function $f:2^{<\mu'}\rightarrow X,\eta\mapsto \eta^{\frown}0^{\mu'}$ and a coloring $c:X\rightarrow S_y(b),\eta\mapsto \tp(a_{\eta}/b)$. Note that $|S_y(b)|\le 2^{|T|}$. By Remark \ref{rem:tree_monotone_decreasing_coloring}, there is $\eta^*\in 2^{<\mu'}$ and $p\in S_y(b)$ such that for all $\nu\in 2^{<\mu'}$ with $\eta^*\lhd \nu$, there is $\rho\in 2^{<\mu'}$ such that
\begin{itemize}
    \item $\nu\lhd \rho$,
    \item for $i=0,1$, there is $\rho^{\frown}\lhd \rho'\in 2^{<\mu'}$ such that $a_{f(\rho')}\models p$.
\end{itemize}
Since $cf(\mu')>\mu+2^{|T|}\geq \mu$, there is an antichain $S\subset 2^{\mu'}$ isomorphic to $A$ such that
\begin{itemize}
    \item for all $\eta'\in S$, $\models \varphi(b,a_{\eta'})$,
    \item for all $\eta_1,\eta_2\in S$, $a_{\eta_1}\equiv_b a_{\eta_2}$.
\end{itemize}
This completes the proof.
\end{proof}

\begin{thm}\label{thm:indisc_ext_NATP}
 Assume $\kappa$ and $\kappa'$ are cardinals such that $2^{|T|}<\kappa<\kappa'$ and cf$(\kappa)=\kappa$. The following are equivalent.
\begin{enumerate}
    \item $T$ is NATP.
    \item For any strongly indiscernible tree $(a_\eta)_{\eta\in2^{<\kappa'}}$ and finite tuple $b$, there are some $\rho\in2^\kappa$ and $b'$ such that
     \begin{enumerate}
         \item $(a_{\rho^\frown0^i})_{i<\kappa'}$ is indiscernible over $b'$,
         \item $b\equiv_{a_\rho}b'$.
     \end{enumerate}
\end{enumerate}
\end{thm}
\begin{proof}
(1$\Rightarrow$2) Suppose $T$ is NATP. Let $(a_\eta)_{\eta\in2^{<\kappa'}}$ and $b$ be given. Define $f:2^\kappa\rightarrow S(b), \eta\mapsto\tp(a_\eta/b)$ and apply Corollary \ref{rem:tree_2kappa_coloring} (b), then there is $S\subseteq 2^\kappa$ such that 
\begin{enumerate}
    \item[(i)] for any $\eta,\nu\in S$, $a_\eta\equiv_b a_\nu$, and
    \item[(ii)] for any $k<\omega$, there exists some tuple in $S$ which is strongly isomorphic to the lexicographic enumeration of $2^k$.
\end{enumerate}
Let $\rho$ be an element in $S$. Define $p(x,y):=$tp$(b,a_\rho)$ and $q(x):=\bigcup_{i<\kappa'}p(x,a_{\rho^\frown0^i})$. We claim $q$ is consistent.\par
Suppose not. Then by compactness and strong indiscernibility, there is $\varphi(x,y)\in p(x,y)$ such that $\{\varphi(x,a_{\rho^\frown0^i}):i<\kappa'\}$ is $k$-inconsistent for some $k<\omega$. On the other hand,  $b$ satisfies $\bigcup_{\eta\in S}p(x,a_\eta)$ by (i), so $b$ also satisfies $\bigwedge_{\eta\in S}\varphi(x,a_\eta)$. Therefore, by (ii) and strong indiscernibility, $(a_\eta)_{\eta\in2^{<\kappa'}}$ and $\varphi$ witness $k$-ATP, which is a contradiction.\par
By Ramsey, compactness, and the claim, there exists $b'\models q(x)$ such that $(a_{\rho^\frown0^i})_{i<\kappa'}$ is indiscernible over $b'$. Note $b'\models q(x)$ implies tp$(ba_\rho)=p=$tp$(b'a_\rho)$.\par

(2$\Rightarrow$1) Suppose $T$ has ATP, witnessed by $(a_\eta)_{\eta\in2^{<\kappa'}}$ and $\varphi(x,y)$. We may assume $(a_\eta)_{\eta\in2^{<\kappa'}}$ is strongly indiscernible.\par
Take a realization $b$ of $\bigwedge_{\eta\in2^\kappa}\varphi(x,a_\eta)$. Then we can find $\rho$ and $b'$ such that $(a_{\rho^\frown0^i})_{i<\kappa'}$ is indiscernible over $b'$ and $b\equiv_{a_\rho}b'$. 
Since $b\models \varphi(x,a_\rho)$, we have $b'\models \varphi(x,a_\rho)$. However, by indiscernibility, $b'$ realizes all the formulas $\varphi(x,a_{\rho^\frown0^i})$ for $i<\kappa$, which contradicts that $\varphi$ has ATP.
\end{proof}

\begin{cor}\label{cor:indisc_ext_NATP}
Let $\kappa$ and $\kappa'$ be cardinals in Theorem \ref{thm:indisc_ext_NATP}. Let $T$ be NATP and a strongly indiscernible tree $(a_\eta)_{\eta\in2^{<\kappa'}}$ and a tuple $b$ is given.
Then for any $S\subseteq 2^\kappa$ satisfying that
\begin{enumerate}
    \item[(i)] for any $\eta,\nu\in S$, $a_\eta\equiv_b a_\nu$, and
    \item[(ii)] for any $k<\omega$, there exists some tuple in $S$ which is strongly isomorphic to the lexicographic enumeration of $2^k$,
\end{enumerate}
and for any $\rho\in2^\kappa$, there is $b'$ satisfying
     \begin{enumerate}
         \item $(a_{\rho^\frown0^i})_{i<\kappa'}$ is indiscernible over $b'$,
         \item $b\equiv_{a_\rho}b'$.
     \end{enumerate}
\end{cor}
\begin{proof}
Clear by the proof of Theorem \ref{thm:indisc_ext_NATP}.
\end{proof}

\noindent By combining the proof of Theorem \ref{one-varable} and Theorem \ref{thm:indisc_ext_NATP}, we have the following.
\begin{thm}\label{thm:improved_indisc_ext_NATP}
 Assume $\kappa$ and $\kappa'$ are cardinals such that $2^{|T|}<\kappa<\kappa'$ and cf$(\kappa)=\kappa$. The following are equivalent.
\begin{enumerate}
    \item $T$ is NATP.
    \item For any strongly indiscernible tree $(a_\eta)_{\eta\in2^{<\kappa'}}$ and $b$ with $|b|=1$, there are some $\rho\in2^\kappa$ and $b'$ such that
     \begin{enumerate}
         \item $(a_{\rho^\frown0^i})_{i<\kappa'}$ is indiscernible over $b'$,
         \item $b\equiv_{a_\rho}b'$.
     \end{enumerate}
\end{enumerate}
\end{thm}

\section{Examples}\label{sec:Algebraic examples}
In this section, we aim to provide criteria for the following algebraic structures to be NATP:
\begin{itemize}
    \item (Mekler's construction for NATP, Theorem \ref{thm:Preservation_of_NATP}) Groups in the groups language.
    \item (The Chatzidakis criterion for NATP, Theorem \ref{thm:PACfield_NATP_via_Galoisgroup}) Fields in the ring language.
    \item (The AKE-principle for NATP, Theorem \ref{thm:AKE_NATP}) Valued fields in the Denef-Pas language.
\end{itemize}

\noindent Also, at the end of the section, we provide structures having ATP. Specifically, we show that the Skolem arithmetic $(\mathbb{N},\cdot)$ and atomless Boolean algebras have ATP.

\subsection{Groups with NATP}\label{Subsection:mekler}

When a structure of countable language is given, Mekler's construction provides a pure group structure preserving some model-theoretic tameness properties of the given structure.
The following properties are known to be preserved by Mekler's construction: stability, simplicity, NIP, $k$-dependence, NTP$_2$, NSOP$_1$, and NSOP$_2$ (cf. \cite{Ahn, Baudisch, CH2, Mekler}).
In this subsection, We aim to prove that NATP is also preserved by the Mekler's construction. We follow the strategy in \cite{Ahn} and \cite{CH2}.

We recall definitions and facts on Mekler's construction based on \cite[Section 2]{CH2} and \cite[Section A.3]{Hod}.
\begin{dfn}\label{dfn:nice_graph}
A graph is called {\it nice} if 
 \begin{enumerate}
     \item for any vertex $a$ and $b$, there exists vertex $c\neq a, b$ such that $c$ is adjacent to $a$ but not to $b$, and
     \item the graph has no cycle of order 3 nor 4.
 \end{enumerate}
\end{dfn}

\begin{fact}\cite[Theorem 5.5.1]{Hod}\label{fact:nice_graph_conversion}
Any structure in a countable language is bi-interpretable with a nice graph.
\end{fact}

Fix an odd prime number $p$.
\begin{dfn}\label{def:Mekler_group}
Given a nice graph $A$, the Mekler group $G(A)$ of $A$ is generated freely in the variety of 2-nilpotent groups of exponent p by the vertices of $A$ by imposing that two generators commute if and only if they are connected by an edge in $A$.
\end{dfn}

Fix a nice graph $A$ and let $G$ be a model of Th$(G(A))$.

\begin{dfn}\label{def:Equi_rel_in_Mekler_group}
Let $g$ and $h$ be elements in $G$. Also, let $C(g), C(h)$ be centralizers of each $g$ and $h$ and $Z$ be the center of $G$.
\begin{enumerate}
    \item By $g\sim h$, we mean $C(g)=C(h)$.
    \item By $g\approx h$, we mean $h=g^rc$ for some $c\in Z$ and non-negative integer $r$.
    \item By $g\equiv_Z h$, we mean $gZ=hZ$.
\end{enumerate}
\end{dfn}

\begin{dfn}
Let $g$ be an element in $G$.
\begin{enumerate}
    \item We say $g$ is isolated if for any $h\in C(g) \cap (G\backslash Z(G))$, $g\approx h$. If not, we say $g$ is non-isolated.
    \item We denote $n(g)$ to be the number of $\approx$-classes in the $\sim$-class of $g$.
\end{enumerate}
\end{dfn}

\begin{fact} \cite[Theorem A.3.10]{Hod}
$G$ can be partitioned into the following five 0-definable sets;
\begin{enumerate}
    \item $Z(G)$, the center of $G$,
    \item $1^\nu$, the set of non-isolated elements $g$ in $G\backslash Z(G)$ such that $n(g)=1$.
    \item $1^\iota$, the set of isolated elements $g$ in $G\backslash Z(G)$ such that $n(g)=1$.
    \item $P$, the set of elements $g\in G$ such that $n(g)=p$, and
    \item the set of elements $g\in G$ such that $n(g)=p-1$.
\end{enumerate}
\end{fact}

\begin{dfn}
Let $g$ be an element in $P$. For $h\in G(A)$, we say $h$ is a handle of $g$ if it is in $1^\nu$ and satisfies $gh=hg$.
\end{dfn}

\begin{fact}
Both $Z(G)$, the center of $G$, and the quotient $G/Z(G)$ are isomorphic to $\mathbb{F}_p$-vector spaces.
\end{fact}

\begin{dfn}
Let $B$ be a subgroup of $G$ containing $Z(G)$. We say a set $C\subseteq G$ is independent modulo $B$ if it is linearly independent over $B$ in terms of the corresponding vector space.
\end{dfn}

\begin{dfn}\label{def:Transversal}
Let $G$ be a model of Th$(G(A))$.
\begin{enumerate}
    \item A $1^\nu$-transversal of $G$, $X^\nu$, is a set containing exactly one representative in each $\sim$-class of $1^\nu$.
    \item An element in $G$ is called proper if it is not a product of elements in $1^\nu$.
    \item A $p$-transversal of $G$, $X^p$, is a set consisting proper elements in $P$ where 
     \begin{enumerate}
      \item for any $g,h\in X^p$, $g$ is not $\sim$-equivalent to $h$,
      \item $X^p$ is maximal with the property that for any finite subset $X'\subseteq X^p$, if all elements of $X'$ have the same handle, then $X'$ is independent modulo the subgroup $\langle Z(G)\cup1^\nu \rangle$.
     \end{enumerate}
    \item A $1^\iota$-transversal of $G$, $X^\iota$, is a set containing exactly one representative in each $\sim$-class of $1^\iota$ and is maximal independent modulo the subgroup $\langle Z(G)\cup1^\nu\cup P \rangle$.
    \item A transversal of $G$ is a union of some $X^\nu$, $X^p$, and $X^\iota$ of $G$.
\end{enumerate}
\end{dfn}

\begin{fact}\label{fact:interpretation_graph_in_Mekler_group}\cite[Theorem A.3.14 (a)]{Hod}
Let $\Gamma$ be an graph interpretation in $G$ such that 
\begin{enumerate}
    \item the universe is $\sim$-equivalent classes of the set of $1^\nu$,
    \item the edge relation is the set of all pairs $([g]_\sim,[h]_\sim)$ where $[g]_\sim\neq[h]_\sim$ and
    $gh=hg$.
\end{enumerate}
Then $\Gamma(G)$ is a model of Th$(A)$.
\end{fact}

\begin{fact}\label{fact:Decomposition_of_Mekler_group}\cite[Theorem A.3.14 (d)]{Hod}
Let $C$ be an infinite nice graph and $G$ be a model of theory of $G(C)$. Then $G$ is isomorphic to $\langle X\rangle \times H$ where $X$ is the transversal of $G$ and $H$ is a subgroup of the center of $G$. Moreover, $H$ is isomorphic to a $\mathbb{F}_p$-vector space.
\end{fact}
We may say $G$ is isomorphic to $\langle X\rangle \times \langle H'\rangle$ where $H'$ is a basis of $H$ in terms of the corresponding vector space. Also note that Th$(H)$ is stable and has elimination of quantifiers.

\begin{fact}\cite[Theorem A.3.14 (c), Corollary A.3.15]{Hod}\cite[Remark 2.12]{CH2}
Let $C$ be an infinite nice graph and $G$ be a saturated and uncountable model of $\Th(G(C))$. For a transversal $X=X^\nu\cup X^p\cup X^\iota$ of $G$, let $G$ be isomorphic to $\langle X\rangle \times H$ for some $H$.
 \begin{enumerate}
     \item Both $\Gamma(G)$ and $H$ are saturated, too.
     \item For any $g\in X^\nu$, the cardinality of the set $\{g'\in X^p:g \text{ is the handle of } g'\}$ is either uncountable or zero.
     \item $|X^\iota|$ is either uncountable or zero.
 \end{enumerate}
\end{fact}

\begin{fact}\label{fact:type-definable_property_of_decomposition}\cite[Proposition 2.18]{CH2}
Let $C$ be an infinite nice graph and $G$ be a model of theory of $G(C)$. 
There exists a partial type $\pi(\bar{x},\bar{y})$ with small tuples $\bar{x},\bar{y}$ such that 
$\bar{a}\bar{b}\models\pi$ if and only if there is a transversal $X$ of $G$ containing $\bar{a}$ and and independent subset $H$ of $Z(G)$ containing $\bar{b}$ such that $G=\langle X\rangle\times\langle H\rangle$.
\end{fact}

\begin{fact}\label{fact:iso_ext_thm_for_Mekler_groups}\cite[Lemma 2.14]{CH2}
Let $C$ be an infinite nice graph and $G$, $X$, and $H$ be given in Fact \ref{fact:Decomposition_of_Mekler_group}. Let $f:Y\rightarrow Z$ be a bijection between two small subsets $Y$ and $Z$ of the transversal $X$ such that
 \begin{enumerate}
     \item $f$ preserves each transversal-type of $Y$ to $Z$.
     \item $f$ preserves the handle relation of $Y$ to $Z$, that is, if $g\in X^\nu$ is the handle of $g'\in X^p$, then $f(g)$ is the handle of $f(g')$, too.
     \item $Y^\nu$ and $Z^\nu$ have the same first order type.
 \end{enumerate}
Then there exist an automorphism $\sigma$ of $G$ extended from $f$. 

Moreover, for any $h,k\in H$, if tp($h$) and tp($k$) are equivalent modulo Th($H$), then we may assume the automorphism $\sigma$ sends $h$ to $k$.
\end{fact}

\begin{thm}\label{thm:Preservation_of_NATP}
For any infinite nice graph $C$, Th$(G(C))$ is NATP if and only if Th$(C)$ is NATP.
\end{thm}
\begin{proof}
The right direction is clear by the interpretability of $C$ in $G(C)$.

Suppose Th$(C)$ is NATP but Th$(G(C))$ has ATP. 
Fix cardinals $\kappa<\kappa'$ with $2^{|T|}<\kappa$, $\kappa=cf(\kappa)$, and $\kappa'=cf(\kappa')$, and assume the antichain tree property is witnessed by a formula $\varphi(x,y)$ and $(a_\eta)_{\eta\in2^{<\kappa'}}$ in some monster model $G$ of Th$(G(C))$. 
The Fact \ref{fact:Decomposition_of_Mekler_group} shows that for a transversal $X$ of $G$, there is $H$ satisfying \ref{fact:Decomposition_of_Mekler_group} such that $G=\langle X\rangle\times\langle H\rangle$.
Thus we can find tuples $\bar{r}_\eta=\bar{r}^\nu_\eta\null^\frown\bar{r}^p_\eta\null^\frown\bar{r}^\iota_\eta\in X$, $\bar{h}_\eta\in H$ and term $t_\eta(y_{r_\eta},y_{h_\eta})$ for each $\eta\in2^{<\kappa'}$ such that $a_\eta=t_\eta(\bar{r}_\eta,\bar{h}_\eta)$.

By Remark \ref{rem:tree_monotone_decreasing_coloring} with the fact that $|T|<\kappa'=cf(\kappa')$, there is an embedding $\iota:2^{<\kappa'}\rightarrow 2^{\kappa'}$ such that 
$t_\eta$ is constant for all $\eta\in \iota[2^{<\kappa'}]$. 
Hence we may assume that for $(a_\eta)_{\eta\in2^{<\kappa'}}$, there exists a term $t(y_r,y_h)$ such that each $a_\eta$ is of the form $t(\bar{r}_\eta,\bar{h}_\eta)$ for some $\bar{r}_\eta=\bar{r}^\nu_\eta\null^\frown\bar{r}^p_\eta\null^\frown\bar{r}^\iota_\eta\in X$ and $\bar{h}_\eta\in H$. We may further assume handles of elements in $\bar{r}^p_\eta$ are placed to the beginning of $\bar{r}^\nu_\eta$ for each $\eta\in2^{<\kappa'}$.

Let $\bar{b}_\eta=\bar{r}_\eta\null^\frown\bar{h}_\eta$ and let $y'$ be a tuple of variables with the length $|\bar{b}_\eta|$ which is constant. If we define $\varphi'(x,y')=\varphi(x,t(y'))$, then this formula and $(b_\eta)_{\eta\in2^{<\kappa'}}$ witness ATP again.

By the strong modeling property (Fact \ref{modeling property}), we may assume $(b_\eta)_{\eta\in2^{<\kappa'}}$ is a strongly indiscernible tree. Also, by Fact \ref{fact:type-definable_property_of_decomposition}, we can ensure that $\bar{r}_\eta=\bar{r}^\nu_\eta\null^\frown\bar{r}^p_\eta\null^\frown\bar{r}^\iota_\eta$ and $\bar{h}_\eta$ are still in transversal $X$ and independent set $H\subseteq Z(G)$ such that $G=\langle X\rangle\times\langle H\rangle$.

Find a realization $c=s(r_c,h_c)$ of $\{\varphi'(x,\bar{b}_\eta):\eta\in2^\kappa\}$ where $s$ is a term and $r_c=r^\nu_c\null^\frown r^p_c\null^\frown r^\iota_c\in X$ and $h_c\in H$. Again, we assume handles of elements in $r^p_c$ placed to the beginning of $r^\nu_c$. 
If we define $\psi(x',y')=\varphi'(s(x'),t(y'))$, then for any $\eta\lhd\eta'\in2^{<\kappa'}$, $\{\psi(x',\bar{b}_\eta),\psi(x',\bar{b}_{\eta'})\}$ is inconsistent and $r_c^\frown h_c$ realizes $\{\psi'(x',\bar{b}_\eta):\eta\in2^\kappa\}$. Moreover, since $r_c^\frown h_c \cap \bigcup\{\bar{b}_\eta:\eta\in2^\kappa\}$ is finite, we may assume the tree is strongly indiscernible over $r_c^\frown h_c \cap \bigcup\{\bar{b}_\eta:\eta\in2^\kappa\}$.

Now consider $r^\nu_c$ and the tree $(\bar{r}^\nu_\eta)_{\eta\in2^{<\kappa'}}$ in $X^\nu$ and $h_c$ and the tree $(\bar{h}_\eta)_{\eta\in2^{<\kappa'}}$ in $H$. 
Note that by Fact \ref{fact:interpretation_graph_in_Mekler_group}, $X^\nu$ can be regarded as a model of Th$(C)$, which is a theory of graph having NATP. Also, since Th$(\langle H\rangle)$ is a theory of vector spaces, it is stable (thus NATP too), and has quantifier elimination.

Let $S(r^\nu_c)$ be the set of all $|\bar{r}^\nu_0|$-types over $r^\nu_c$ and $S(h_c)$ be the set of all $|\bar{h}_0|$-types over $h_c$. Let $f:2^{<\kappa'}\rightarrow S(r^\nu_c)\times S(h_c)$ be a function such that $f(\eta)=(\text{tp}(\bar{r}^\nu_\eta/r^\nu_c),\text{tp}(\bar{h}_\eta/h_c))$. By Corollary \ref{rem:tree_2kappa_coloring}, we can find $S\subseteq 2^\kappa$ such that 
\begin{enumerate}
    \item[(i)] for any $\eta,\eta'\in S$, $\bar{r}^\nu_\eta\equiv_{r^\nu_c} \bar{r}^\nu_{\eta'}$,
    \item[(ii)] for any $\eta,\eta'\in S$, $\bar{h}_\eta\equiv_{h_c} \bar{h}_{\eta'}$, and
    \item[(iii)] for any $k<\omega$, there exists some tuple in $S$ which is strongly isomorphic to the lexicographic enumeration of $2^k$,
\end{enumerate}
Fix an element $\rho$ in $S$. 
Then by Corollary \ref{cor:indisc_ext_NATP}, there exists $r'^\nu_c$ such that $r_c^\nu\equiv_{b_\rho}r'^\nu_c$ and $(\bar{r}^\nu_{\rho^\frown0^i})_{i<\kappa'}$ is indiscernible over $r'^\nu_c$.
We do the similar work for $\langle H\rangle$, then find some $h'_c$ such that $h_c\equiv_{h_\rho}h'_c$ and $(\bar{h}_{\rho^\frown0^i})_{i<\kappa'}$ is indiscernible over $h'_c$.

By Fact \ref{fact:iso_ext_thm_for_Mekler_groups} and the strong indiscernibility of $(\bar{b}_\eta)_{\eta\in2^{<\kappa'}}$ over $r_c^\frown h_c \cap \bigcup\{\bar{b}_\eta:\eta\in2^\kappa\}$, we can find a handle preserving bijection and then extend it to an automorphism of $G$ to have
that tp($r_c h_c/\bar{r}_\rho\bar{h}_\rho$)=tp($r'_c h'_c/\bar{r}_\rho\bar{h}_\rho$) and tp($\bar{r}_\rho\bar{h}_\rho/r'_c h'_c$)= tp($\bar{r}_{\rho^\frown0}\bar{h}_{\rho^\frown0}/r'_c h'_c$). So, we have $G\models\psi(r'_c h'_c,\bar{r}_\rho\bar{h}_\rho)\wedge\psi(r'_c h'_c,\bar{r}_{\rho^\frown0}\bar{h}_{\rho^\frown0})$, which is a contradiction.
\end{proof}

We will later observe that there is an NATP theory of countable language which is SOP$_1$ and TP$_2$ (See Example \ref{ex:example_sop_TP2_NATP}).
Thus, by Fact \ref{fact:nice_graph_conversion}, and the preservation of NSOP$_1$, TP$_2$, and NATP in Mekler's construction, we can obtain an NATP pure group having SOP$_1$ and TP$_2$.

\subsection{Fields with NATP}\label{Subsection:Chat-Ramsey_PAC_fields}
We recall several facts on PAC fields, which will be used to give a criteria for PAC fields to be (N)ATP in this subsection. A field $K$ is called \textit{pseudo algebraically closed (PAC)} if any absolutely irreducible variety over $K$ has a $K$-rational point. We consider a PAC field as a $\CL_{ring}$-structure for the ring language $\CL_{ring}=\{+,-,\times;0,1\}$. And we only consider fields contained in a fixed algebraically closed field $\Omega$ of big enough cardinality. For a subfield $K$ of $\Omega$, we write $K^s$ for the separable closure of $K$ in $\Omega$ and $G(K)$ for the absolute Galois group, $G(K^s/K)$, of $K$.

For PAC fields, having NATP is a non-trivial condition, that is, there is a PAC field whose theory has ATP. 
We recall the existence of a PAC field encoding a given graph $(V,E)$ where $E$ is a binary irreflexive symmetric relation on a set $V$.
\begin{fact}\label{fact:graph_coding_PACfield}\cite[Subection 28.20]{FJ}
Let $(V,E)$ be a graph. Then, there is a PAC field $K$ and $\CL_{ring}$-formulas $\Phi_V$ such that $(\Phi_V(K),\Phi_E(K))$ is a graph isomorphic to $(V,E)$.
\end{fact}

\begin{remark}\label{rem:PACfield_with_ATP}
There is a PAC field whose theory has ATP.
\end{remark}
\begin{proof}
There is a structure $M$ of finite language whose theory has ATP ({\it e.g.} $(\BN,\cdot)$, see Example \ref{ex:Skolem arithmetic ATP}) and there is a graph $(V,E)$ which is bi-interpretable with $M$. By Fact \ref{fact:graph_coding_PACfield}, there is a PAC field $K$ interpreting the graph $(V,E)$ so that $\Th(K)$ has ATP.
\end{proof}

Next, we give a sufficient condition for PAC fields to have NATP via their Galois groups. We recall the notion of the complete system of a profinite group ({\it cf.} \cite[Subsection 2.1]{C2}). Let $\CL_{G}$ be a $\omega$-sorted language consisted with the following relations:
\begin{itemize}
    \item a binary relation $\le_{n,m}$ for $n\ge m$;
    \item a binary relation $C_{n,m}$ for $n\ge m$;
    \item a ternary relation $P_n$ for $n\in \omega$.
\end{itemize}
If there is no confusion, we write $\le$, $C$, and $P$ for $\le_{n,m}$, $C_{n,m}$ and $P_n$ respectively. For a profinite group $G$, the \textit{complete system} of $G$ is given as follows:
\begin{itemize}
    \item For each $n\in \omega$, the sort $n$ is the disjoint union of $G/N$ for each open normal subgroup $N$ of $G$ with $[G:N]\le n+1$ so that the sort $0$ consists of the single element $G$.
    \item The relations $\le$, $C$, and $P$ are interpreted as follows:
    \begin{itemize}
        \item For $gN\in n$ and $hM\in m$, $$gN\le hM \Leftrightarrow N\subset M.$$
        \item For $gN\in n$ and $hM\in m$, $$C(gN,hM)\Leftrightarrow gN\subset hM.$$
        \item For $g_1N_1,g_2N_2,g_3N_3\in n$, $$P(g_1N_1,g_2N_2,g_3N_3)\Leftrightarrow N_1=N_2=N_3(=:N),\ g_1g_2N=g_3N.$$
    \end{itemize}
\end{itemize}

\noindent A subset $S_0$ of a complete system $S$ is called a \textit{subsystem} of $S$ if it satisfies the following:
\begin{enumerate}
    \item $\forall \sigma,\tau\in S_0,\exists \rho\in S_0(\rho\le \sigma\wedge \rho\le \tau)$;
    \item $\forall \sigma\in S_0,\forall \tau\in S(\sigma\le \tau\rightarrow \tau\in S_0)$.
\end{enumerate}
For a subset $A$ of $S$, we write $\langle A\rangle $ for the subsystem of $S$ generated by $A$. For a field $K$, we write $SG(K)$ for the complete system of $G(K)$. Let $\varphi:K^s\rightarrow L^s$ be an isomorphism with $\varphi[K]=L$. Then, $\varphi$ induces an isomorphsim $$\Phi:G(L)\rightarrow G(K),\sigma\mapsto \varphi^{-1}\circ \sigma\circ \varphi.$$ This isomorphism induces an isomorphism $S\Phi:SG(K)\rightarrow SG(L)$, called the \textit{double dual} of $\varphi$.

Using complete systems, we have the following type-amalgamation for PAC fields, proved by Chatzidakis.

\begin{fact}\label{fact:zoe_type-amalgamation_PACfield}\cite[Theorem 3.1]{C2}
Let $F$ be a PAC field, and let $E,A,B,C_1,C_2$ be algebraically closed subsets of $F$, with $E$ contained in $A,B,C_1,C_2$. Assume that $A\cap B=E$, that $A$ and $C_1$, and $B$ and $C_2$ are SCF-independent over $E$ and that if the degree of imperfection of $F$ is finite, then $E$ contains a $p$-basis of $F$. Moreover, assume that there is an $E^s$-isomorphism $\varphi:C_1^s\leftrightarrow C_2^s$ such that $\varphi(C_1)=C_2$, and that there is $S_0\subset SG(F)$, and elementary (in $SG(F)$) isomorphisms $$S\Psi_1:\langle SG(C_1),SG(A)\rangle\rightarrow \langle S_0,SG(A)\rangle,$$ $$S\Psi_2:\langle SG(C_2),SG(B)\rangle\rightarrow \langle S_0,SG(B)\rangle$$
such that
\begin{enumerate}
    \item $S\Psi_1$ is the identity on $SG(A)$, $S\Psi_2$ is the identity on $SG(B)$, and $$S\Psi_i(SG(C_i))=S_0\,;$$
    \item if $S\Phi:SG(C_1)\rightarrow SG(C_2)$ is the morphsim double dual to $\varphi$, then $$S\Psi_2\circ S\Phi=S\Psi_1\restriction_{SG(C_1)}.$$
\end{enumerate}
Then, in some elementary extension $F^*$ of $F$, there is $C$ whcih is SCF-independent from $AB$ over $E$ realizes $\tp(C_1/A)\cup \tp(C_2/B)$ and with $SG(C)=S_0$ (the variables for $\tp(C_1/A)$ and $\tp(C_2/B)$ are identified via $\varphi$).
\end{fact}

\begin{theorem}\label{thm:PACfield_NATP_via_Galoisgroup}
For a PAC field $F$, $\Th(F)$ has NATP if $\Th(SG(F))$ has NATP.
\end{theorem}
\begin{proof}
We follow the proof scheme of \cite[Proposition 7.2.8]{Ram1}. We assume that $F$ is a monster model of $\Th(F)$ and fix a Skolemization of $F$ so that $\acl=\dcl$ and any definably closed subset is an elementary submodel. In the proof, indiscernibility will always mean with respect to this structure, but $\acl$ means model-theoretic algebraic closure in $F$ in the ring language. Note that the complete system $SG(F)$ is uniformly interpretable in $(F^s,F)$ ({\it cf.} \cite[Fact 7.1.7]{Ram1}).

Suppose $\Th(F)$ has ATP, witnessed by a formula $\varphi(x;y)$ via the parameters $(a_{\eta})_{\eta\in 2^{<\mu}}$, which forms a strongly indiscernible tree, for large enough cardinal $\mu$. By Remark \ref{rem:arbitrary_configuration_embedding_ATP} and by taking an automorphism, there is $b$ such that for an antichain $X=\{0^{i\frown}1:i<\omega+\omega\}\cup\{0^{\omega+\omega\frown}\eta:\eta\in 2^{\omega}\}\subset 2^{<\mu}$,
\begin{itemize}
    \item for all $\eta'\in X$, $\models \varphi(b;a_{\eta'})$;
    \item for all $\eta_1,\eta_2\in X$, $a_{\eta_1}\equiv_b a_{\eta_2}$.
\end{itemize}
\noindent Note that $(a_{0^{i\frown}1})_{i<\omega+\omega}$ is $\{a_{0^{\omega+\omega\frown}\eta}:\eta\in 2^{\omega}\}$-indiscernible. Put $a_i:=a_{0^{i\frown}1}$ for each $i<\omega+\omega$. By Ramsey, compactness, and automorphism, we may assume that $(a_i)_{i<\omega+\omega}$ is $\{b\}\cup\{a_{0^{\omega+\omega\frown}\eta}:\eta\in 2^{\omega}\}$-indiscernible.

Let $E$ be the Skolem hull of $(a_i)_{i<\omega}$ in $F$. Let $B:=\acl(bE)$, and let $A_{\eta}:=\acl(a_{0^{\omega\frown}\eta} E)$ for each $\eta\in 2^{\le\omega+\omega}$. Put $A_i:=\acl(a_{0^{\omega+i\frown}1})$ for each $i<\omega$. Then, we have that
\begin{itemize}
    \item $(A_{\eta})_{\eta\in 2^{\le\omega+\omega}}$ is strongly indiscernible over $E$;
    \item $(A_i)_{i<\omega}$ is $B\cup \bigcup\{A_{0^{\omega\frown}\eta:\eta\in 2^{\omega}}\}$-indiscernible;
    \item $(A_i)_{i<\omega}$ is an $E$-finitely satisfiable Morley sequence, enumerated in reverse.
\end{itemize}
By Kim's lemma in the stable theory $SCF$, we have that $B$ is $SCF$-independent from $A_0$ over $E$. By strongly indiscernibility, $(A_i)_{i<\omega}$ is also $A_{\emptyset}$-indiscernible, and so $A_0$ is $SCF$-independent from $A_{\emptyset}$ over $E$.

Choose $B'$ so that $A_{\emptyset}B'\equiv_E A_0B$. Let $q(X;SG(A_0)):=\tp(SG(B)/SG(A_0))$. Then, we have that
\begin{itemize}
    \item[$(\dagger)_1$] for each $\eta\in 2^{\omega}$, $SG(A_{0^{\omega\frown}\eta})\equiv_{SG(B)}SG(A_0)$ ;
    \item[$(\dagger)_2$] for any finitely many incomparable vertices $\eta_1,\ldots,\eta_n\in 2^{\omega+\omega+\omega}$, there are $\eta_1',\ldots\eta_n'\in 2^{\omega}$ such that $$SG(A_{\eta_1})\ldots SG(A_{\eta_n})\equiv SG(A_{0^{\omega\frown}\eta_1'})\ldots SG(A_{{\omega\frown}\eta_n'});$$
     \item[$(\dagger)_3$] $(SG(A_{\eta}))_{\eta\in 2^{\le\omega+\omega}}$ is a strongly indiscernible tree over $SG(E)$.
\end{itemize}
We have that $q(X;SG(A_0))\cup q(X;SG(A_{\emptyset}))$ is consistent. Namely, suppose that $q(X;SG(A_0))\cup q(X;SG(A_{\emptyset}))$ is inconsistent. So, there is a formula $\Phi(X;Y)$ such that $\Phi(X;SG(A_0))\in q(X;SG(A_0))$ and $\{\Phi(X;SG(A_0)),\Phi(X;SG(A_{\emptyset}))\}$ is inconsistent $(\ddagger)$. By $(\dagger)_1$ and $(\dagger)_2$, for any antichain $C\subset2^{\le \omega+\omega}$, $\{\Phi(X;SG(A_{\eta})):\eta\in C\}$ is consistent. Therefore, by $(\dagger)_3$ and $(\ddagger)$, the formula $\Phi(X;Y)$ witnesses ATP via $(SG(A_{\eta}))_{\eta\in 2^{\le\omega+\omega}}$, which contradicts that $\Th(SG(F))$ has NATP. Let $S_0\models q(X;SG(A_0))\cup q(X;SG(A_{\emptyset}))$. By Fact \ref{fact:zoe_type-amalgamation_PACfield}, there is $B_*(=\acl(b_*E))\subset F$ such that $B_*\equiv_{A_0} B$ and $B_*\equiv_{A_{\emptyset}}B'$, and $SG(B_*)=S_0$. By the choice of $B_*$ and definition of $A_0$ and $A_{\emptyset}$, we have that $\models \varphi(b_*;a_{0^{\omega\frown}1})\wedge\varphi(b_*;a_{0^{\omega}})$, which is a contradiction. Thus, $\Th(F)$ has NATP.
\end{proof}

\begin{remark}\label{rem:PAC_structure_ATP}
It is known that for $n\ge 1$, a PAC field $F$ has NSOP$_n$ if the complete system $SG(F)$ has NSOP$_n$ ({\it cf.} \cite[Theorem 3.9]{C2} and \cite[Corollary 7.2.7, Proposition 7.2.8]{Ram1}). It is generalized to PAC structures in \cite{HL}. The criterion for a PAC field in Theorem \ref{thm:PACfield_NATP_via_Galoisgroup} can also be generalized to PAC structures by the same argument.
\end{remark}

We end this subsection with a question on a relationship between pseudo real closed (PRC) fields, pseudo $p$-adic (PpC) fields, and NATP. It is well known that a PAC field is simple if and only if it is bounded (cf. \cite{CP98, C97}), and Montenegro in \cite{Mon17} proved that a PRC field or PpC field is NTP$_2$ if it is bounded, and the converse holds for PRC fields.  Also. Kaplan and Ramsey in \cite{KR20} showed that a Frobenius field is NSOP$_1$. Here, a Frobenius field is a PAC field whose Galois group has the embedding property (see \cite[Chapter 24]{FJ}). According to the equation ``NATP=NTP$_2$+NSOP$_1$", we ask the following question.
\begin{question}\label{question:PRC_PpC_NATP}
Let $F$ be a PRC field or a PpC field whose the Galois group $G(F)$ has the embedding property. Does $\Th(F)$ have NATP?
\end{question}

\subsection{Valued field with NATP}\label{Subsection:AKE_NATP}
We recall several facts on valued fields, which will be used to give a criteria for a valued field to be NATP.

Let $\mathcal{K}=(K,\Gamma,k,\nu:K\rightarrow \Gamma,\ac:K\rightarrow k)$ be a henselian valued field of characteristic $(0,0)$ in the {\bf Denef-Pas language} $\CL_{Pas}=\CL_{K}\cup\CL_{\Gamma,\infty}\cup\CL_{k}\cup\{\nu,\ac\}$, where $\CL_{\Gamma,\infty}$ is the language of ordered abelian group expanded by a constant symbol $\infty$ (see \cite[Definition 1.3.1]{Sinclair}). The following properties hold: 
\begin{itemize}
	\item For $a\in K$, $$\nu(a)=\infty\Leftrightarrow \ac(a)=0\Leftrightarrow a=0;$$
	
	\item For $a,b\in K$, suppose $\nu(a)=\nu(b)\neq \infty$. Then, $$\big[\ac(a)+\ac(b)=\ac(a+b)\wedge \ac(a+b)\neq 0 \big]\Leftrightarrow \big[\nu(a)=\nu(b)=\nu(a+b)\big];$$
	
	\item For $a,b\in K$, suppose $\nu(a)<\nu(b)$. Then, $\ac(a)=\ac(a+b)$.

\end{itemize}
Suppose that $\Th(\mathcal K)$ admits \textit{relative quantifier elimination}, that is, every formula $\varphi(x,x^{\Gamma},x^{k})$ is equivalent to one of the form
$$\bigvee_{i\le n}\chi_i\left(\nu(f_1(x)),\ldots,\nu(f_m(x),x^{\Gamma})\right)\wedge \rho_i\left(\ac(f_1(x)),\ldots,\ac(f_m(x),x^{k} \right)$$
where $x$, $x^{\Gamma}$, $x^k$ are tuples of variables corresponding to $K$, $\Gamma$, $k$ respectively, $\chi_i\in \CL_{\Gamma}$, $\rho_i\in \CL_{k}$, and $f_j$ are polynomials in $\BZ[x]$. For example, any henselian valued fields of characteristic $(0,0)$ and algebraically maximal Kaplansky fields of characteristic $(p,p)$ give the theory having relative quantifier elimination (see \cite[page 9]{Sinclair}). We aim to prove that $\Th(\mathcal K)$ is NATP if and only if $\Th(k)$ is NATP. We follow the strategy of the proof that $\Th(\mathcal K)$ is NTP$_2$ if and only if $\Th(k)$ is NTP$_2$ in \cite[Subsection 7.2]{Che}.

\begin{remark}\label{rem:path_consistent_NATP}
Let $T$ be a complete theory having NATP. Let $\varphi(x;y)$ be a formula and let $(a_{\eta})_{\eta\in \omega^{\le \omega}}$ be strongly indiscernible over $\emptyset$. Suppose $\{\varphi(x;a_{\eta}):\eta\in \omega^{\omega}\}$ is consistent. Then, for each $\eta\in \omega^{\omega}$, $\{\varphi(x;\eta\restriction_i):i\le \omega\}$ is consistent.
\end{remark}
\begin{proof}
Suppose there is $\eta\in \omega^{\omega}$ such that $\{\varphi(x;\eta\restriction_i):i\le \omega\}$ is inconsistent. Then, by strong indiscernibility, $(a_{\eta})_{\eta\in \omega^{<\omega}}$ witnesses $k$-ATP of $\varphi(x;y)$ for some $k\ge 2$, which contradicts the assumption that $T$ has no ATP.
\end{proof}

\begin{fact}\label{fact:basic_facts_on_PAS language}
\begin{enumerate}
	\item Let $\varphi(x)$ be a formula with parameters for $x$ corresponding to $K$ and $|x|=1$. Then, $\varphi(x)$ is equivalent to a finite disjunction of formulas of the form $$\chi_i(\nu(x-d_{i,1}),\ldots,\nu(x-d_{i,m}),d_i^{\Gamma})\wedge \rho_i(\nu(x-d_{i,1}),\ldots,\nu(x-d_{i,m}),d_i^{k}),$$ where $\chi_i\in\CL_{\Gamma,\infty}$ and $\rho_i\in \CL_{k}$ ({\it cf.} \cite[Proposition 2.1.3]{Sinclair}).
	\item $\Gamma$ and $k$ are stably embedded ({\it cf.} \cite[Lemma 2.3]{CH}).
\end{enumerate}
\end{fact}

\begin{remark}\label{rem:induced_structure_ATP}
Let an $\emptyset$-definable set $D$ be stably embedded and assume that $D_{ind}$ is NATP. Let $\varphi(x;y)$ be a formula such that $\models \forall y(\varphi(x;y)\rightarrow D(x))$. Then, $\varphi(x;y)$ has NATP.
\end{remark}
\begin{proof}
Suppose there is $(a_{\eta})_{\eta\in 2^{<\kappa}}$ witnessing ATP of $\varphi(x;y)$. Note that each $a_{\eta}$ can be outside $D$. We may assume that $(a_{\eta})_{\eta\in 2^{<\kappa}}$ is strongly indiscernible and $cf(\kappa)>|\CL|$.

Since $D$ is stably embedded, for each $\eta\in \kappa$, there is $\psi_{\eta}(x;y)\in \CL$ and $b_{\eta}\in D$ such that $$\varphi(x;a_{\eta})\wedge D(x)=\psi_{\eta}(x;b_{\eta})\wedge D(x).$$ By Remark \ref{rem:tree_monotone_decreasing_coloring} and $cf(\kappa)>|\CL|$, there are a formula $\psi(x;y)$ and an embedding $\iota:2^{<\omega}\rightarrow 2^{<\kappa}$ such that for each $\eta\in 2^{<\omega}$, $$\psi_{\iota(\eta)}(x;z)=\psi(x;z).$$ So, we may assume that there are a strong indiscernible $(a_{\eta})_{\eta\in 2^{<\omega}}$ witnessing ATP of $\varphi(x;y)$ and a formula $\psi(x;z)$ such that for each $\eta\in 2^{<\omega}$, there is $b_{\eta}\in D$ such that $$\varphi(x;a_{\eta})\wedge D(x)=\psi(x;b_{\eta})\wedge D(x).$$ Then, the formula $\psi(x;z)\wedge D(x)$ has ATP witnessed by $(b_{\eta})_{\eta\in 2^{<\omega}}$ and so $D_{ind}$ has ATP, a contradiction.
\end{proof}

\begin{fact}\label{fact:basic_on_valuation_indiscernibles}\cite[Lemma 7.9, Lemma 7.10]{Che}
Let $(c_i)_{i\in I}$ be indiscernible.
\begin{enumerate}
	\item Consider the function $(i,j)\mapsto \nu(c_j-c_i)$ for $i<j$. Then, it satisfies one of the following:
	\begin{itemize}
		\item It is strictly increasing depending on $i$ so the sequence is pseudo-convergent.
		\item It is strictly decreasing depending on $j$ so the sequence taken in the reverse direction is pseudo-convergent. In this case, we call such a sequence \normalfont{decreasing}.
		\item It is constant. In this case, we call such a sequence \textit{constant}.

	\end{itemize}
	
	\item Suppose that $(c_i)_{i\in I}$ is a pseudo-convergent sequence. Let $\bar I$ be the Dedekind closure of $I$. Then, for any $a$, there is some $h\in \bar I\cup \{-\infty,\infty\}$ such that after taking $c_{-\infty}$ and $c_{\infty}$ such that $c_{-\infty}^{\frown}(c_i)_{i\in I}^{\frown}c_{\infty}$ is indiscernible,
	\begin{itemize}
		\item[For $i<h$:] $$\nu(c_{\infty}-c_i)<\nu(a-c_{\infty}),\ \nu(a-c_i)=\nu(c_{\infty}-c_i),\ \ac(a-c_i)=\ac(c_{\infty}-c_i),$$
		\item[For $i>h$:] $$\nu(c_{\infty}-c_i)>\nu(a-c_{\infty}),\ \nu(a-c_i)=\nu(a-c_{\infty}),\ \ac(a-c_i)=\ac(a-c_{\infty}).$$
	\end{itemize}
\end{enumerate}
\end{fact}
 
From now on, we assume that $\Th(k)$ has no ATP (automatically, $\Th(\Gamma)$ has no ATP because it has NIP).
\begin{lemma}\label{lem:NATP_complexity_one}
Let $$\varphi(x;yy^{\Gamma}y^k)\equiv\chi(\nu(x-y);y^{\Gamma})\wedge \rho(\ac(x-y);y^k)$$ where $\chi\in \CL_{\Gamma,\infty}$ and $\rho\in\CL_k$. Then, $\varphi(x;yy^{\Gamma}y^k)$ does not witness ATP.
\end{lemma}
\begin{proof}
Suppose $\varphi(x;yy^{\Gamma}y^k)$ witnesses ATP with strongly indiscernible $(d_{\eta})_{\eta\in \omega^{\le \omega+\omega}}$, where $d_{\eta}=c_{\eta}d_{\eta}^{\Gamma}d_{\eta}^k$, and $c_{\eta}\in K$ corresponds to $y$, $d_{\eta}^{\Gamma}\in \Gamma$ corresponds to $y^{\Gamma}$, and $d_{\eta}^k\in k$ corresponds to $y^k$.

We choose universal antichains $X_i$ in $\omega^{\le \omega+\omega}$  for each $i\in \omega\cup\{-\infty,\infty\}$ as follows:
\begin{itemize}
	\item $X_{-\infty}:=\{0^{\frown}\eta:\eta\in \omega^\omega\}$, and $X_{\infty}:=\{1^{\omega\frown}\eta:\eta\in \omega^{\omega}\}$,
	\item for each $i\in \omega$, $X_i:=\{1^{1+i\frown }0^{\frown}\eta:\eta\in \omega^{\omega}\}$,
	\item $(X:=)\bigcup_{i\in\omega\cup\{-\infty,\infty\}} X_i$ is an antichain.
\end{itemize}
We choose a subtree $\mathcal T_{i}$ in $\omega^{\le \omega+\omega}$  for each $i\in \omega\cup\{-\infty,\infty\}$ as follows:
\begin{itemize}
	\item $\mathcal T_{-\infty}:=\{0^{\frown}\eta:\eta\in \omega^{\le \omega}\}$, and $\mathcal T_{\infty}:=\{1^{\omega\frown}\eta:\eta\in \omega^{\le \omega}\}$,
	\item for each $i\in \omega$, $\mathcal T_i:=\{1^{1+i\frown }0^{\frown}\eta:\eta\in \omega^{\le\omega}\}$.
\end{itemize}
Then for each choice of $\eta_i,\eta_i'\in\mathcal T_i$ for all $i\in \omega\cup\{-\infty,\infty\}$, the sequence $(d_{\eta_i})_{i\in\omega\cup\{-\infty,\infty\}}$ is indiscernible and $(d_{\eta_i})_{i\in\omega\cup\{-\infty,\infty\}}\equiv (d_{\eta_i'})_{i\in\omega\cup\{-\infty,\infty\}}$. Let $a\models \{\varphi(x;d_{\eta}):\eta\in X\}$.

\medskip

\noindent{\it Case 1.} There is $\eta_i\in X_i$ for each $i\in \omega$ such that $(c_{\eta_i})_{i\in\omega}$ is pseudo-convergent and so for each $i\in \omega$ for any $\eta_i\in X_i$, $(c_{\eta_i})_{i\in\omega}$ is pseudo-convergent. Fix $\eta_{-\infty}\in X_{-\infty}$ and $\eta_{\infty}\in X_{\infty}$ arbitrarily. Let $h\in \{-\infty\}\cup \omega\cup\{\infty\}$ be as given by Fact \ref{fact:basic_on_valuation_indiscernibles}(2).

\smallskip

\noindent{\it Subcase 1.1.} Assume $0<h$. Then, $$\nu(a-c_{\eta_0})=\nu(c_{\eta_\infty}-c_{\eta_0}),\ \ac(a-c_{\eta_0})=\ac(c_{\eta_{\infty}}-c_{\eta_0}).$$ Then, $c_{\eta_{\infty}}\models \varphi(x;d_{\eta_0})$. Since $\mathcal T_0$ is strongly indiscernible over $c_{\eta_{\infty}}$, $$c_{\eta_{\infty}}\models \{\varphi(x;d_{\zeta}):\zeta\unlhd \eta_0,\zeta\in \mathcal T_0\},$$ which is impossible.

\smallskip

\noindent{\it Subcase 1.2.} Therefore we may assume that $h\in \{-\infty,0\}$. 
Then, we have that $$\nu(a-c_{\eta_i})=\nu(a-c_{\eta_\infty}),\ \ac(a-c_{\eta_i})=\ac(a-c_{\eta_{\infty}}),\ \nu(a-c_{\eta_{\infty}})<\nu(c_{\eta_{\infty}}-c_{\eta_i})$$ 
for each $0<i\in \omega$.

Let $$\chi'(x^{\Gamma};e_{\eta}^{\Gamma})\equiv \chi(x^{\Gamma};d_{\eta}^{\Gamma})\wedge x^{\Gamma}<\nu(c_{\eta_\infty}-c_{\eta})$$ with $e_{\eta}^{\Gamma}=d_{\eta}^{\Gamma}\coc \nu(c_{\eta_{\infty}}-c_{\eta})$ for each $\eta\in \omega^{\le \omega}$. 
Then, we have that $$\nu(a-c_{\eta_{\infty}})\models \{\chi'(x^{\Gamma};e_{\eta}^{\Gamma}):\eta\in X_1\},\ \ac(a-c_{\eta_{\infty}})\models \{\rho(x^k;d_{\eta}^k):\eta\in X_1\}.$$ Since $\Th(\Gamma)$ and $\Th(k)$ has no ATP, by Remark \ref{rem:path_consistent_NATP}, for any $\eta\in X_1$, there are $a_{\eta}^{\Gamma}\in \Gamma$ and $a_{\eta}^k\in k$ such that $$a_{\eta}^{\Gamma}\models \{\chi'(x^{\Gamma};e_{\zeta}^{\Gamma}):\zeta\unlhd \eta, \zeta\in \mathcal T_1\},\ a_{\eta}^k\models \{\rho(x^k;d_{\zeta}^k):\zeta\unlhd \eta,\zeta\in \mathcal T_1\}.$$ For each $\eta\in X_1$, take $a_{\eta}\in K$ such that $$\nu(a_{\eta}-c_{\eta_{\infty}})=a_{\eta}^{\Gamma},\ \ac(a_{\eta}-c_{\eta_{\infty}})=a_{\eta}^k.$$ Then, for each $\eta\in X_1$, we have that $a_{\eta}\models\{\varphi(x;d_{\zeta}):\zeta\unlhd \eta,\zeta\in \mathcal T_1\}$, a contradiction.

\medskip

\noindent{\it Case 2.} There is $\eta_i\in X_i$ for each $i\in \omega$ such that $(c_{\eta_i})_{i\in\omega}$ is decreasing, which is reduced to the first case by reversing the order $\{-\infty\}\cup\omega\cup\{\infty\}$.

\medskip

\noindent{\it Case 3.} There is $\eta_i\in X_i$ for each $i\in \omega$ such that $(c_{\eta_i})_{i\in\omega}$ is constant.

\smallskip

\noindent{\it Subcase 3.1.} There is $i\in \omega$ such that $\nu(a-c_{\eta_i})<\nu(c_{\eta_{\infty}}-c_{\eta_i})$. By the assumption, we have that $$\nu(c_{\eta_{\infty}}-c_{\eta_i})=\nu(c_{\eta_{\infty}}-c_{\eta})=\nu(c_{\eta}-c_{\eta_i})$$ for any $\eta\in X_{i+1}$. So, we have that $$\nu(a-c_{\eta_i})=\nu(a-c_{\eta})=\nu(a-c_{\eta_{\infty}}),\ \ac(a-c_{\eta})=\ac(a-c_{\eta_{\infty}})$$ for any $\eta\in X_{i+1}$. Then, by the argument in the subcase 1.2., for each $\eta\in X_{i+1}$, there is $a_{\eta}\models \{\varphi(x;d_{\zeta}):\zeta\unlhd \eta,\zeta\in \mathcal T_{i+1}\}$.

\smallskip

\noindent{\it Subcase 3.2.} For every $i\in \omega$ and for every $\eta \in X_i$, $$\nu(a-c_{\eta})\ge \nu(c_{\eta_{\infty}}-c_{\eta}).$$ Note that there can be at most one $i\in \omega$ such that $\nu(a-c_{\eta_i})>\nu(c_{\eta_{\infty}}-c_{\eta_i})$. Namely, if $\nu(a-c_{\eta_j})>\nu(c_{\eta_{\infty}}-c_{\eta_j})$ for some $j\neq i\in \omega$, then we have that 
\begin{align*}
\nu(c_{\eta_{\infty}}-c_{\eta_j})&=\nu(c_{\eta_j}-c_{\eta_i})\\
&\ge \min\{\nu(a-c_{\eta_i}),\nu(a-c_{\eta_j})\}\\
&>\nu(c_{\eta_{\infty}}-c_{\eta_i})=\nu(c_{\eta_{\infty}}-c_{\eta_j}),
\end{align*}
a contradiction. So, there is $i_0\in \omega$ such that for all $i> i_0$, for all $\eta\in X_i$, $$\nu(a-c_{\eta})=\nu(c_{\eta_{\infty}}-c_{\eta})=\nu(c_{\eta}-c_{\eta_{i_0}})=\nu(a-c_{\eta_{i_0}})(\neq \infty).$$ Then, for $\eta\in X_{i_0+1}$, $\ac(a-c_{\eta})\neq \ac(c_{\eta_{i_0}}-c_{\eta})$ so that $$\ac(a-c_{\eta})=\ac(a-c_{\eta_{i_0}})+\ac(c_{\eta_{i_0}}-c_{\eta})\neq 0.$$ Namely, if $\ac(a-c_{\eta})=\ac(c_{\eta_{i_0}}-c_{\eta})$, then $$\nu(a-c_{\eta_{i_0}})=\nu\left((a-c_{\eta})-(c_{\eta_{i_0}}-c_{\eta})\right)>\nu(a-c_{\eta})=\nu(a-c_{\eta_{i_0}}),$$ a contradiction.

For each $\eta\in \mathcal T_{i_0+1}$, let $$\rho'(x^{k},e_{\eta}^k)\equiv\rho(x^k+\ac(c_{\eta_{i_0}}-c_{\eta});d_{\eta}^k)\wedge x^k+ \ac(c_{\eta_{i_0}}-c_{\eta})\neq 0$$ with $e_{\eta}^k:=d_{\eta}^k\coc\ac(c_{\eta_{i_0}}-c_{\eta})$. Since $\ac(a-c_{\eta_{i_0}})\models \{\rho'(x^k;e_{\eta}^k):\eta\in X_{i_0+1}\}$ and $\mathcal T_{i_0+1}$ is strongly indiscernible over $c_{\eta_{i_0}}$, by Remark \ref{rem:path_consistent_NATP}, for each $\eta\in X_{i_0+1}$, there is $a_{\eta}^k\in k$ such that $$a_{\eta}^k\models \{\rho'(x^k;e_{\zeta}^k):\zeta\unlhd \eta,\zeta\in \mathcal T_{i_0+1}\}.$$ For each $\eta\in X_{i_0+1}$, take $a_{\eta}\in K$ such that $\nu(a_{\eta}-c_{\eta_{i_0}})=\nu(a-c_{\eta_{i_0}})\wedge \ac(a_{\eta}-c_{\eta_{i_0}})=a_{\eta}^k$. Then, for each $\eta\in X_{i_0+1}$ and for each $\zeta\unlhd \eta\in \mathcal T_{i_0+1}$, we have that 
$$\nu(a_{\eta}-c_{\eta_{i_0}})=\nu(a-c_{\eta_{i_0}})=\nu(c_{\eta_{i_0}}-c_{\eta})=\nu(c_{\eta_{i_0}}-c_{\zeta}),$$ and $$\ac(a_{\eta}-c_{\zeta})=\ac(a_{\eta}-c_{\eta_{i_0}})+\ac(c_{\eta_{i_0}}-c_{\zeta})=a_{\eta}^k+\ac(c_{\eta_{i_0}}-c_{\zeta})\neq 0.$$
So, $$\nu(a_{\eta}-c_{\zeta})=\nu(c_{\eta_{i_0}}-c_{\zeta})=\nu(a-c_{\eta_{i_0}}),$$ which implies that $a_{\eta}\models \{\varphi(x;d_{\zeta}):\zeta\unlhd \eta,\zeta\in \mathcal T_{i_0+1}\}$, a contradiction.
\end{proof}

\begin{lemma}\label{lem:NATP_complexity_n}
Let $$\varphi(x;y)=\chi(\nu(x-y_1),\ldots,\nu(x-y_n),y^{\Gamma})\wedge \rho(\ac(x-y_1),\ldots,\ac(x-y_n),y^k)$$ be a formula where $\chi\in \CL_{\Gamma,\infty}$ and $\rho\in \CL_{k}$. Then, any tree $(d_{\eta})_{\eta\in \omega^{<\omega}}$ does not witness ATP of $\varphi(x;y)$.
\end{lemma}
\begin{proof}
We prove it by induction on $n$. The case that $n=1$ is done by Lemma \ref{lem:NATP_complexity_one}. Suppose it holds for $n-1$ and it does not hold for $n$. Let $(d_{\eta})_{\eta\in \omega^{\le \omega}}$ witness ATP of $\varphi(x;y)$. We may assume that $(d_{\eta})_{\eta\in \omega^{\le \omega}}$ is strongly indiscernible. For each $\eta\in \omega^{\le \omega}$ Write $d_{\eta}=d_{\eta}^1\ldots d_{\eta}^n d_{\eta}^{\Gamma}d_{\eta}^k$ where each $d_{\eta}^i\in K$, $d_{\eta}^{\Gamma}\subset \Gamma$, and $d_{\eta}^k\subset k$. Note that $d_{\eta}^1\neq d_{\eta}^n$ for each $\eta\in \omega^{\le \omega}$. Let $a\models \{\varphi(x;d_{\eta}):\eta\in \omega^{\omega}\}$. We may assume $(d_\eta)_{\eta\in\omega^\omega}$ is $\delta$-indiscernible over $a$. Then $(d_\eta)_{\eta\in\omega^\omega}$ satisfies exactly one of the following cases.

\medskip

\noindent{\it Case 1.} $\nu(a-d_{\eta}^1)<\nu(d_{\eta}^1-d_{\eta}^n)$ for all $\eta\in\omega^\omega$.

\noindent{\it Case 2.} $\nu(a-d_{\eta}^1)>\nu(d_{\eta}^n-d_{\eta}^1)$ for all $\eta\in\omega^\omega$.

\noindent{\it Case 3.} $\nu(a-d_{\eta}^n)<\nu(d_{\eta}^n-d_{\eta}^1)$ for all $\eta\in\omega^\omega$.

\noindent{\it Case 4.} $\nu(a-d_{\eta}^n)<\nu(d_{\eta}^n-d_{\eta}^1)$ for all $\eta\in\omega^\omega$.

\noindent{\it Case 5.} $\nu(a-d_{\eta}^1)=\nu(a-d_{\eta}^n)=\nu(d_{\eta}^n-d_{\eta}^1)(\neq \infty)$ for all $\eta\in\omega^\omega$.

\medskip

Suppose $(d_\eta)_{\eta\in\omega^\omega}$ satisfies Case 1. Then, $$\nu(a-d_{\eta}^1)=\nu(a-d_{\eta}^n),\ \ac(a-d_{\eta}^1)=\ac(a-d_{\eta}^n)$$ for all $\eta\in\omega^\omega$. Let
\begin{align*}
\varphi_1(x;y)&\equiv \chi(\nu(x-y_1),\ldots,\nu(x-y_{n-1}),\nu(x-y_1),y^{\Gamma})\\
&\wedge \nu(x-y_1)<\nu(y_n-y_1)\\
&\wedge \rho(\ac(x-y_1),\ldots,\ac(x-y_{n-1}),\ac(x-y_1),y^k).
\end{align*}
Clearly $\{\varphi_1(x;d_\eta):\eta\in\omega^\omega\}$ is consistent. Note that $$\nu(x-y_1)<\nu(y_n-y_1)\Rightarrow \nu(x-y_1)=\nu(x-y_n).$$ Thus $\{\varphi_1(x;d_{0^i}):i\in\omega\}$ is not consistent. So, $\varphi_1$ witnesses ATP and it violates the inductive hypothesis.

By the same argument, for Case 2, 3, 4, use $\varphi_2$, $\varphi_3$, $\varphi_4$ respectively to derive a contradiction, where
\begin{align*}
\varphi_2(x;y)&\equiv \chi(\nu(x-y_1),\ldots,\nu(x-y_{n-1}),\nu(y_n-y_1),y^{\Gamma})\\
&\wedge \nu(x-y_1)>\nu(y_n-y_1)\\
&\wedge \rho(\ac(x-y_1),\ldots,\ac(x-y_{n-1}),\ac(y_n-y_1),y^k),\\
\varphi_3(x;y)&\equiv \chi(\nu(x-y_n),\nu(x-y_2),\ldots,\nu(x-y_n),y^{\Gamma})\\
&\wedge \nu(x-y_n)<\nu(y_n-y_1)\\
&\wedge \rho(\ac(x-y_n),\ac(x-y_2),\ldots,\ac(x-y_n),y^k),\\
\varphi_4(x;y)&\equiv \chi(\nu(y_n-y_1),\nu(x-y_2),\ldots,\nu(x-y_n),y^{\Gamma})\\
&\wedge \nu(x-y_n)>\nu(y_n-y_1)\\
&\wedge \rho(\ac(y_1-y_n),\ac(x-y_2),\ldots,\ac(x-y_n),y^k).
\end{align*}
\noindent For Case 5, note that $\nu(a-d_{\eta}^1)=\nu(a-d_{\eta}^n)=\nu(d_{\eta}^n-d_{\eta}^1)(\neq \infty)$ implies that $\ac(a-d_{\eta}^n)=\ac(a-d_{\eta}^1)-\ac(d_{\eta}^n-d_{\eta}^1)(\neq 0)$. By using
\begin{align*}
\varphi_5(x;y)&\equiv \chi(\nu(x-y_1),\ldots,\nu(x-y_{n-1}),\nu(y_n-y_1),y^{\Gamma} )\\
&\wedge \nu(x-y_1)=\nu(y_n-y_1)\\
&\wedge \rho(\ac(x-y_1),\ldots,\ac(x-y_{n-1}),\ac(x-y_1)-\ac(y_n-y_1),y^k)\\
&\wedge \ac(x-y_1)-\ac(y_n-y_1)\neq 0,
\end{align*}
we can derive a contradiction as in the Case 1. This completes the proof.
\end{proof}

\begin{theorem}\label{thm:AKE_NATP}
$\Th(\mathcal K)$ has NATP if and only if $\Th(k)$ has NATP.
\end{theorem}
\begin{proof}
It is enough to show the direction of the right to left. Suppose $\Th(k)$ has NATP. By Fact \ref{fact:basic_facts_on_PAS language}(2), Remark \ref{rem:induced_structure_ATP}, and Theorem \ref{one-varable}, it is enough to show that any formula $\varphi(x;\y)$ does not witness ATP when $x$ corresponds to $K$ and $|x|=1$.

Suppose there is a formula $\varphi(x;y)$ witnessing ATP with $x$ corresponding to $K$ and $|x|=1$. Let $\kappa$ be a cardinal whose cofinality is uncountable. Let $(c_{\eta})_{\eta\in \omega^{<\kappa}}$ witness ATP of $\varphi(x;y)$. Consider a coloring $c:\omega^{<\kappa}\rightarrow \CL_{Pas}$ given as follows: For $\eta\in \omega^{<\kappa}$, $c(\eta)=\psi(x;z)$ is a formula such that for some $e$, $\varphi(x;c_{\eta})$ is equivalent to $\psi(x;d)$, which is of the form $$\bigvee_{i<n}\chi_i(\nu(x-d_{i,1}),\ldots,\nu(x-d_{i,m}),d_i^{\Gamma})\wedge \rho_i(\nu(x-d_{i,1}),\ldots,\nu(x-d_{i,m}),d_i^{k}),$$ where $\chi_i\in\CL_{\Gamma,\infty}$ and $\rho_i\in \CL_{k}$. Such a coloring $c$ is well-defined by Fact \ref{fact:basic_facts_on_PAS language}(1). By Remark \ref{rem:tree_monotone_decreasing_coloring}, there is an embedding $\iota:\omega^{<\omega}\rightarrow \omega^{<\kappa}$ such that $(c_{\iota(\eta)})_{\eta\in \omega^{<\omega}}$ witnesses ATP of $\psi(x;z)$, which is impossible by Lemma \ref{lem:disjunction_NATP} and by Lemma \ref{lem:NATP_complexity_n}.
\end{proof}

\begin{example}\label{ex:example_sop_TP2_NATP}
Let $F$ be a $\omega$-free PAC field of characteristic $0$ and let $\Gamma$ be an ordered abelian group. Then $\Th(F)$ in the ring language is NSOP$_1$ ({\it cf.} \cite[Corollary 6.2]{CR}) and TP$_2$ ({\it cf.} \cite[Section 3.5]{C08}), and $\Th(\Gamma)$ is NIP and SOP. Consider the Hahn field $K:=F[[t^{\Gamma}]]$, which is a Henselian valued field of characteristic $(0,0)$. Then, by Theorem \ref{thm:AKE_NATP}, $\Th(K)$ is NATP even though it has TP$_2$ by the residue field and SOP by the value group.
\end{example}

\subsection{Examples of ATP}
In this subsection, we provide examples having ATP.
\begin{example}\label{ex:Skolem arithmetic ATP}
Skolem arithmetic $(\mathbb{N},\cdot)$ has ATP. Thus, any theory interpreting Skolem arithmetic such as $(\mathbb{Z},+,\cdot,0,1)$ and ZFC also has ATP.
\end{example}
\begin{proof}
We start with the following observation.
\begin{claim}\label{claim:key_observation_Skolem_arithmetic}
Let $I$ be a finite set of indices and $S$ be a nonempty class of subsets of $I$ such that $\emptyset\notin S$ and for any nonempty $J\subseteq J'$ in $\mathcal{P}(I)$, if $J'\in S$ then $J\in S$. Let $$\varphi(x,y)\equiv x\neq 1\wedge x|y$$ be the formula saying $x\neq 1$ and $x$ divides $y$. Then there exists $(a_i)_{i\in I}$ in $\mathbb{N}$ such that 
 \begin{enumerate}
     \item for any $J\in S$, $\{\varphi(x,a_j):j\in J\}$ is consistent, and
     \item for any nonempty $J\in \mathcal{P}(I)\backslash S$, $\{\varphi(x,a_j):j\in J\}$ is inconsistent.
 \end{enumerate}
\end{claim}
\begin{proof}
Let $\{J_0,J_1,\ldots,J_{m-1}\}$ be an enumeration of $S$.
Let $p_n$ be the $n$th prime number and $a_i$ be the product of all $p_n$'s such that $i\in J_n$. If $i$ is in none of $J_n\in S$ then we set $a_i=1$.
Note that $p_n$ divides $a_i$ if and only if $J_n\ni i$.
We claim that this $(a_i)_{i\in I}$ witnesses the given condition.

Choose any set $J_n$ in $S$. Then for any $i\in J_n$, $p_n$ is a prime factor of $a_i$.
Thus $\{\varphi(x,a_j):j\in J\}$ is realized by $p_n$.
On the other hand, choose any nonempty set $J$ not in $S$ and suppose $N\in\mathbb{N}$ realizes the set $\{\varphi(x,a_j):j\in J\}$. Since $N\neq1$, we can choose a prime factor $p$ of $N$, then $p$ is a prime factor of $a_j$ for all $j\in J$ so there is some $n<m$ such that $p=p_n$. Then $J$ is a subset of $J_n$ and so $J\in S$, which is a contradiction.
\end{proof}
\noindent Let $I=2^{<n}$ for some $n<\omega$. Let $S$ be the set of all antichains in $I$. Clearly, any subset of antichain is also an antichain. By Claim \ref{claim:key_observation_Skolem_arithmetic} and compactness, the formula $\varphi$ witnesses ATP.
\end{proof}

\begin{example}\label{ex:atomless_BA_ATP}
The theory of atomless Boolean algebras has ATP.
\end{example}
\begin{proof}
Let $\mathcal B=(B,\wedge,\vee,0,1)$ be an atomless Boolean algebra and let $\le$ be the partial order on $B$ defined as follow: For $x,y\in B$, $$x\le y\Leftrightarrow x\wedge y=x.$$
We want to show $$\varphi(x,y)\equiv x\neq 0\wedge x\le y$$ witnesses ATP.

Let $0<n<\omega$ be some natural number and $\{J_0,J_1,\ldots,J_{m-1}\}$ be an enumeration of all antichains in $2^{<n}$.
Choose $p_0,\ldots,p_{m-1}$ be non-zero elements in $B$ such that for all $i<j<m$, $p_i\wedge p_j=0$ ($\ast$) (This is possible because $\mathcal B$ is atomless).
For each $\eta\in 2^{<n}$, let $I_\eta=\{i<m:\eta\in J_i\}$ and let $a_\eta$ be the disjunction of all $p_i$'s such that $i\in I_\eta$. 
Note that $p_i\le a_\eta$ if and only if $J_i\ni \eta$.
We claim that this $(a_\eta)_{\eta\in 2^{<n}}$ witnesses the ATP conditions.

For any antichain $J_n$ in $2^{<n}$, by definition, $p_n$ realizes $\{\varphi(x,a_\eta):\eta\in J_n\}$.
On the other hand, let $\eta$ and $\nu$ be indices in $2^{<n}$ such that $\eta\unlhd\nu$, and suppose $\varphi(x,a_\eta)\wedge\varphi(x,a_\nu)$ is realized by some $b\in B$. 
Then
$$b=b\wedge a_\eta\wedge a_\nu=b\wedge\bigvee_{i\in I_\eta}p_i\wedge\bigvee_{i\in I_\nu}p_i=b\wedge\bigvee_{i\in I_\eta,j\in I_\nu}(p_i\wedge p_j)$$ 
is not zero, so by ($\ast$), there should be some $i<m$ both in $I_\eta$ and $I_\nu$. 
Then by definition, both $a_\eta$ and $a_\nu$ is in an antichain $J_i$, a contradiction.

Compactness finishes the proof.
\end{proof}

\begin{remark}
 The same argument in Example \ref{ex:atomless_BA_ATP} shows that the formula $\varphi$ witnesses ATP in infinite atomic Boolean algebras, too.
\end{remark}

\end{document}